\documentclass[10pt, a4paper, reqno, english]{amsart}

\usepackage[utf8]{inputenc}
\usepackage{hyperref}
\usepackage{amssymb}
\usepackage{enumitem}

\usepackage{a4wide}

\newcommand\ZZ{\mathbb{Z}}
\newcommand\NN{\mathbb{N}}
\newcommand\CC{\mathbb{C}}
\newcommand\RR{\mathbb{R}}
\newcommand\QQ{\mathbb{Q}}

\renewcommand\AA{\mathbb{A}}

\newcommand{\OO}{\mathcal{O}}
\newcommand{\K}{K}
\newcommand{\Ok}{\OO_\K} 
\newcommand{\units}{{\Ok^\times}}

\newcommand{\id}[1]{\mathfrak{#1}}
\newcommand{\ppp}{\id p}
\newcommand{\aaa}{\id a}
\newcommand{\bbb}{\id b}

\newcommand{\ddd}{\id d}
\newcommand{\fff}{\id f}
\newcommand{\places}{{\Omega_\K}}
\newcommand{\archplaces}{{\Omega_\infty}}
\newcommand{\nonarchplaces}{{\Omega_0}}
\newcommand{\abs}[1]{\left|#1\right|}
\newcommand{\absv}[1]{\left|#1\right|_v}
\newcommand{\normv}[1]{\left\lVert #1 \right\rVert_v}
\DeclareMathOperator{\N}{\mathfrak{N}}

\newcommand\dd{\,\mathrm{d}} 

\newcommand\xx{\mathbf{x}} 

\newcommand\where{\, : \,}

\DeclareMathOperator{\vol}{vol}

\DeclareMathOperator{\Tr}{Tr}

\DeclareMathOperator{\rank}{rank}

\newcommand\R{R} 
\newcommand\B{\mathcal{B}} 
\newcommand\va{\boldsymbol{\alpha}} 
\newcommand\q{q} 
\newcommand\n{\mathfrak{n}} 
\newcommand\vn{\boldsymbol{\nu}}
\newcommand\vt{\boldsymbol{\theta}}
\newcommand\vv{\mathbf{v}}
\newcommand\vE{\mathbf{E}}
\newcommand\vw{\mathbf{w}}
\newcommand\vy{\mathbf{y}}
\newcommand\vh{\mathbf{h}}

\newcommand\vu{\mathbf{u}}
\newcommand\vX{\mathbf{X}}
\newcommand\vY{\mathbf{Y}}

\newcommand\vW{\mathbf{W}}
\newcommand\vU{\mathbf{U}}
\newcommand\vH{\mathbf{H}}
\newcommand\vg{{\boldsymbol{\gamma}}}

\newcommand{\M}{\mathfrak{M}}
\newcommand{\m}{\mathfrak{m}}
\newcommand{\card}[1]{\# #1}
\newcommand{\PP}{\mathbb{P}}

\newtheorem{theorem}{Theorem}
\newtheorem{lemma}[theorem]{Lemma}

\newtheorem{corollary}[theorem]{Corollary}
\theoremstyle{definition}

\newtheorem*{acknowledgements}{Acknowledgements}

\numberwithin{theorem}{section}
\numberwithin{equation}{section}

\begin{document}

\setcounter{tocdepth}{1}

\title[Forms over number fields]{Forms of differing degrees over number fields}

\author{Christopher Frei}
\address[C. Frei]{Technische Universität Graz,
 Institut für Analysis und Computational Number Theory, Steyrergasse 30/II, A-8010 Graz, Austria}
\email{frei@math.tugraz.at}

\author{Manfred Madritsch}
\address[M. Madritsch]{1. Universit\'e de Lorraine, Institut Elie Cartan de Lorraine, UMR 7502, Vandoeuvre-l\`es-Nancy, F-54506, France;\newline
2. CNRS, Institut Elie Cartan de Lorraine, UMR 7502, Vandoeuvre-l\`es-Nancy, F-54506, France}
\email{manfred.madritsch@univ-lorraine.fr}

\date{\today}

\begin{abstract}
  Consider a system of polynomials in many variables over the ring of
  integers of a number field $K$. We prove an asymptotic formula for
  the number of integral zeros of this system in homogeneously
  expanding boxes. As a consequence, any smooth and geometrically
  integral variety $X\subseteq \PP_K^{m}$ satisfies the Hasse
  principle, weak approximation and the Manin-Peyre conjecture, if
  only its dimension is large enough compared to its degree.

  This generalizes work of Skinner, who considered the case where all
  polynomials have the same degree, and recent work of Browning and
  Heath-Brown, who considered the case where $K=\QQ$. Our main tool is
  Skinner's number field version of the Hardy-Littlewood circle
  method. As a by-product, we point out and correct an error in
  Skinner's treatment of the singular integral.
\end{abstract}


%
%


\maketitle

\tableofcontents

\section{Introduction}\label{sec:introduction}
\subsection{Main result}
Let $K$ be a number field of degree $n$ over $\QQ$. We consider a
system of polynomials in $s$ variables over the ring of integers $\Ok$
of $K$ and let $D$ be the maximum of their degrees. We assume the polynomials
to be ordered by their degrees, that is, for each $d \in \{1, \ldots,
D\}$, we are given polynomials
\begin{equation*}
  G_{d,1}, \ldots, G_{d,t_d} \in \OO_K[x_1, \ldots, x_s]
\end{equation*}
of degree $d$, where $t_d \geq 0$ with $t_D\geq1$. The total number of polynomials is
$T := t_1 + \cdots + t_D$. We are interested in quantitative
statements about the common zeros of these polynomials.

To this end, we fix an integral ideal $\n$ of $\Ok$ and a $\ZZ$-basis
$\omega_1, \ldots, \omega_n$ of $\n$. We will also consider $\omega_1,
\ldots, \omega_n$ as an $\RR$-basis of $V := K \otimes_\QQ \RR$. By a
\emph{box $\mathcal{B}$ aligned to the basis}, we mean the set of all
$\xx = (x_1, \ldots, x_s) \in V^s$, where each $x_i$ has the form
$x_{i,1}\omega_1 + \cdots + x_{i,n}\omega_n$, such that the
coordinates $\vX = (x_{i,j})_{i,j}\in\RR^{ns}$ lie in a given box $B
\subseteq [-1,1]^{ns}$ with sides aligned to the coordinate axes of
$\RR^{ns}$. Given such a box $\mathcal{B}$, we study the asymptotics
of the counting function
\begin{equation*}
  N(P) := \card\{\xx \in \n^s\cap P \mathcal{B}\where G_{d,i}(\xx)=0 \text{ for all }1\leq d \leq D, 1\leq i \leq t_d\}.
\end{equation*}
This counting function is a classical object of interest and has been
investigated in many special cases. To name a few, Birch
\cite{MR0150129} considered forms over $\QQ$ with all degrees
equal. Skinner \cite{MR1446148} generalized Birch's result to
arbitrary number fields $K$. Schmidt \cite{MR781588} and Browning and
Heath-Brown \cite{BrowningHeathBrown2014} considered forms over $\QQ$
whose degrees might be different.

The main purpose of this article is to generalize the work of Browning
and Heath-Brown \cite{BrowningHeathBrown2014} to arbitrary number
fields, just as Skinner did with Birch's work. In addition, we fix 
an error in Skinner's paper \cite{MR1446148}, see
Subsection \ref{subsec:singint}. To state the main result, we need to
introduce some further notation. Let
\begin{equation*}
  \Delta := \{ 1\leq d\leq D \where t_d \geq 1\}.
\end{equation*}
For each $1\leq d \leq D$, $1\leq i\leq t_d$, let $F_{d,i}$ be the leading form of the polynomial $G_{d,i}$, that is, the degree-$d$-part of $G_{d,i}$. For each degree $d\in\Delta$, we consider the $(t_d\times s)$-Jacobian matrix
\begin{equation*}
  J_d(\xx):=
  \begin{pmatrix}
    \nabla F_{d,1}(\xx)\\
    \vdots\\
    \nabla F_{d,t_d}(\xx)
  \end{pmatrix}
\end{equation*}
corresponding to the leading forms of degree $d$ and the affine
variety $S_d \subset \mathbb{A}^s_\K$ defined by the condition
$\rank(J_d(\xx))<t_d$. We define $B_d$ to be the dimension of $S_d$,
and $B_d := 0$ if $d\notin\Delta$. In the following, we always
assume that $B_d < s$.

Let $\mathcal{D}_0:=0$ and, for $1\leq d \leq D$,
\begin{equation*}
  \mathcal{D}_d := t_1 + 2 t_2 + \cdots + d t_d,
\end{equation*}
so that $\mathcal{D}:=\mathcal{D}_D$ is the sum of the degrees of all
our polynomials $G_{d,i}$. Write
\begin{equation}
  \label{eq:def_sd}
  s_d := \sum_{k=d}^D\frac{2^{k-1}(k-1)t_k}{s-B_k}.
\end{equation}
Our main theorem is as follows.

\begin{theorem}\label{thm:main}
  Assume that
  \begin{equation}\label{eq:main_hyp}
\mathcal{D}_d\left(\frac{2^{d-1}}{s-B_d}+s_{d+1}\right)+s_{d+1}+\sum_{j=d+1}^Ds_jt_j < 1
  \end{equation}
  holds for all $d \in \Delta\cup\{0\}$. Then there is a positive
  $\delta$, such that
  \begin{equation*}
    N(P) = \mathfrak{S}\mathfrak{J}\cdot P^{n(s-\mathcal{D})} + O(P^{n(s-\mathcal{D})-\delta})
  \end{equation*}
  for $P\to \infty$. Here, $\mathfrak{S}$ is the usual singular series
  and $\mathfrak{J}$ is the usual singular integral. The implicit
  constant in the error term may depend on $K$, the polynomials
  $G_{d,i}$, $\n$, the basis $\omega_1, \ldots, \omega_n$, and the box
  $\mathcal{B}$, but not on $P$ and the positive constant $\delta$
  depends on $K$ and the systems of forms.
\end{theorem}
More precisely, the positive constant $\delta$ in the exponent of the
error term may be chosen depending only on $n,T$, and the margin by
which the inequalities \eqref{eq:main_hyp} are satisfied.

We describe later in this introduction what is meant by the usual
singular series and the usual singular integral. Theorem
\ref{thm:main} is a number field version of \cite[Theorem
1.2]{BrowningHeathBrown2014}. One should note that our hypotheses
\eqref{eq:main_hyp} are exactly the ones used by Browning and
Heath-Brown over $\QQ$. In particular, they are independent of the
degree $n$ of $K$.

In the special case where all degrees
are equal to $D$, our hypotheses \eqref{eq:main_hyp} simplify to
Skinner's condition
\begin{equation}\label{eq:skinner_condition}
  s - B_D > t_D(t_D+1)(D-1)2^{D-1}
\end{equation}
from \cite{MR1446148}, which is the same one as Birch's
\cite{MR0150129}.

Our proof relies on Skinner's number field version of the
Hardy-Littlewood circle method \cite{MR1290198, MR1446148}. There are
only very few other applications of the circle method over number
fields whose results are even close to the best available results over
$\QQ$, and in particular do not depend on the degree $n$ of $K$. See,
for example, \cite{MR1446148,MR3188633,MR3229043}.

\subsection{Consequences}
Here we collect some consequences of Theorem \ref{thm:main}. They are
number field versions of the results stated in the introduction of
\cite{BrowningHeathBrown2014}. We omit most of the proofs, since they
are almost verbatim the same as in \cite{BrowningHeathBrown2014}. The
following corollary provides simpler hypotheses that imply those of
Theorem \ref{thm:main}.

\begin{corollary}
  Let $B_{\max} := \max\{B_d\where d\in\Delta\}$ and
  \begin{align*}
    u_d &:= \sum_{k=d}^D2^{k-1}(k-1)t_k \quad\text{ for }\quad 1\leq d \leq D+1,\\
    s_0(d) &:= \mathcal{D}_d(2^{d-1}+u_{d+1}) + u_{d+1} + \sum_{j=d+1}^Du_jt_j\\
    s_0 &:= \max\{s_0(d)\where d\in\Delta\cup \{0\}\}.
  \end{align*}
  Then the conclusion of Theorem \ref{thm:main} holds whenever $s > B_{\max} +s_0$. 
\end{corollary}

Of course, the explicit bounds for $s_0$ in case of systems of
quadratic forms and a form of higher degree, computed in
\cite[Corollary 1.4]{BrowningHeathBrown2014}, and in case of systems
consisting of two forms of differing degrees, computed in
\cite[Corollary 1.5]{BrowningHeathBrown2014} are still valid in our
number field version. Moreover, \cite[Theorem
1.6]{BrowningHeathBrown2014} provides us with the bounds
\begin{equation*}
  s_0 + T - 1 \leq \mathcal{D}^22^{D-1}\leq T^2D^22^{D-1}\\
\end{equation*}
and
\begin{equation}\label{eq:s0_bound}
  s_0 + T - 1 \leq (\mathcal{D}-1)2^{\mathcal{D}}.
\end{equation}
By the last inequality, the condition $s>B_{\max}+s_0$ in the corollary is implied by the original condition \eqref{eq:skinner_condition} of Birch and Skinner in case of a single form of degree $D$.

In the following, we will specialize our results to non-singular systems of leading forms. To the system of forms $F_{d,j}(\xx)$,
$1\leq d \leq D, 1\leq j \leq t_d$, we associate the
$(T\times s)$-Jacobian matrix $J(\xx)$ formed by the partial
derivatives of all $T$ forms $F_{d,j}$ with with respect to all $s$
variables $x_i$. We call the system $(F_{d,j})_{d,j}$ of forms  \emph{non-singular}, if $\rank J(\xx) = T$ for every nonzero
$\xx \in \overline{\QQ}^s$ satisfying $F_{d,j}(\xx) = 0$ for all
$d,j$.

Following \cite{BrowningHeathBrown2014}, we define two systems of
forms $(F_{d,j})_{d,j}$ and $(\tilde{F}_{d,j})_{d,j}$, with
$F_{d,j},\tilde{F}_{d,j} \in \Ok[x_1,\ldots, x_s]$ of degree $d$, to be
\emph{equivalent} if for every pair $(d,j)$ we have
\begin{equation}\label{eq:forms_equivalent}
  \tilde{F}_{d,j} = F_{d,j} - \sum_{i<j}H_{d,i}F_{d,i} - \sum_{e<d}\sum_{i\leq  t_e}H_{e,i}F_{e,i},
\end{equation}
with forms $H_{e,i} \in \Ok[x_1, \ldots, x_s]$ of degree
$d-e$. Moreover, we define a system $(F_{d,j})_{d,j}$ of forms to be
\emph{optimal}, if for any choice of $(d,i)$, the sub-system
\begin{equation*}
  \{F_{d,j}\where j \geq i\} \cup \{F_{e,j}\where d<e\leq D, j\leq t_e\}
\end{equation*}
is nonsingular. In \cite[Section 3]{BrowningHeathBrown2014} it is shown that every nonsingular system of forms is equivalent to an optimal system, and that every optimal system $(F_{d,i})_{d,i}$ satisfies
\begin{equation*}
  B_d \leq t_d + \cdots + t_D - 1 \quad\text{ for all }\quad 1\leq d \leq D,
\end{equation*}
and hence in particular $B_{\max} \leq T - 1$.

Assume we are given a system of polynomials $(G_{d,j})_{d,j}$, with
$G_{d,j}\in \Ok[x_1, \ldots, x_s]$ of degree $d$, with leading forms
$(F_{d,j})_{d,j}$, and a system of forms $(\tilde{F}_{d,j})_{d,j}$
equivalent to $(F_{d,j})_{d,j}$. Applying the expression for
$\tilde{F}_{d,j}$ in \eqref{eq:forms_equivalent} to the polynomials
$G_{d,j}$ instead of the forms $F_{d,j}$, we can easily write down a
system of polynomials $(\tilde{G}_{d,j})_{d,j}$ with leading forms
$(\tilde{F}_{d,j})_{d,j}$, such that the $\tilde{G}_{d,j}$ generate
the same ideal of $\Ok[x_1, \ldots, x_s]$ as the $G_{d,j}$. Hence, we
may replace every system $(G_{d,j})_{d,j}$ with a non-singular system of
leading forms $(F_{d,j})_{d,j}$ by a system $(\tilde{G}_{d,j})_{d,j}$
with an optimal system of leading forms $(\tilde{F}_{d,j})_{d,j}$, and
having the same common zeros.

Together with \eqref{eq:s0_bound}, this allows us to deduce the following generalization of \cite[Theorem 1.7]{BrowningHeathBrown2014}.

\begin{theorem}\label{thm:nonsing_system}
  Suppose that our system of leading forms $(F_{d,j})_{d,j}$ is non-singular and satisfies $s > (\mathcal{D}-1)2^{\mathcal{D}}$. Then there is a positive $\delta$ such that
  \begin{equation*}
    N(P) = \mathfrak{S}\mathfrak{J}\cdot P^{n(s-\mathcal{D})} + O(P^{n(s-\mathcal{D})-\delta}).
  \end{equation*} 
Here, $\mathfrak{J}>0$ whenever the system of equations
\begin{equation*}
  F_{d,j}(\xx)=0 \quad\text{ for }\quad 1\leq d \leq D, 1\leq j \leq t_d
\end{equation*}
has a nonzero solution in the interior of $\mathcal{B}\subset
V$. Moreover, $\mathfrak{S}>0$ whenever the system of equations
\begin{equation*}
  G_{d,j}(\xx)=0 \quad\text{ for }\quad 1\leq d\leq D, 1\leq j \leq t_d
\end{equation*}
has a nonzero solution in the completion $n_\ppp^s \subset
(\Ok)_\ppp^s$ for every prime ideal $\ppp$ of $\Ok$.
\end{theorem}

The conditions under which $\mathfrak{J}>0$ and $\mathfrak{S}>0$
follow from well known facts about $\mathfrak{J}$ and $\mathfrak{S}$
and from the fact that, under the hypotheses of Theorem \ref{thm:nonsing_system}, the system of leading forms $(F_{d,j})_{d,j}$ defines a smooth complete intersection (see
\cite[Lemma 3.2]{BrowningHeathBrown2014}). In particular, the singular series $\mathfrak{S}$ has the usual interpretation as a product of local densities.

Theorem \ref{thm:nonsing_system} has far-reaching consequences for
smooth projective complete intersections. In fact, every smooth
complete intersection $X \subseteq \PP^{s-1}_K$ is defined by a
non-singular system of forms $(F_{d,i})_{d,i}$, and if $X$ is
non-degenerate (not contained in a proper linear subspace of
$\PP^{s-1}_K$), then $\deg(X) \geq \mathcal{D}$.

It is known that an asymptotic formula as in Theorem
\ref{thm:nonsing_system} implies the Hasse principle and weak
approximation for $X$, see \cite{MR1446148}. We can also say something
about the Manin-Peyre conjecture \cite{MR974910,MR1340296}. Let
$\places$ be the set of all places of $K$, and for each $v\in\places$
let $n_v := [K_v : \QQ_v]$ be the local degree at $v$. Let
$\normv{\cdot}$ be any norm on $K_v^{s}$, coinciding with the usual
$\max$-norm if $v$ is non-archimedean. Then
\begin{equation}\label{eq:height}
  H((x_1 : \cdots : x_s)) := \prod_{v\in\places}\normv{(x_1, \ldots, x_s)}^{n_v (s - \mathcal{D})}
\end{equation}
defines an anticanonical height function on the rational points
$X(K)$. The proof of \cite[Theorem 4.8]{Loughran2014} shows that the
conclusion of Theorem~\ref{thm:nonsing_system} implies the Manin-Peyre
conjecture for $X$ with respect to the height $H$.

Thus, every smooth and non-degenerate complete intersection
$X \subseteq \PP_K^{s-1}$ with
\begin{equation*}
  s > (\deg(X)-1)2^{\deg X}
\end{equation*}
satisfies the Hasse principle, weak approximation, and the Manin-Peyre
conjecture for the anticanonical heights defined above.

Browning and Heath-Brown show in \cite{BrowningHeathBrown2014} that
every smooth and geometrically integral variety $X \subseteq
\PP^{m}_\QQ$ satisfying
  \begin{equation}\label{eq:general_degree_cond}
    \dim(X) \geq (\deg(X)-1)2^{\deg(X)}-1
  \end{equation}
is already a complete intersection. Their arguments
hold as well over arbitrary number fields, which provides us with the
following nice consequence of Theorem \ref{thm:nonsing_system},
generalizing \cite[Theorem 1.1]{BrowningHeathBrown2014}.

\begin{theorem}
  Let $X \subseteq \PP^{m}_K$ be a smooth and geometrically integral
  variety satisfying \eqref{eq:general_degree_cond}. Then $X$
  satisfies the Hasse principle, weak approximation, and the
  Manin-Peyre conjecture with respect to the height functions defined
  in \eqref{eq:height}.
\end{theorem}

\subsection{The circle method over number fields}
Our proof of Theorem \ref{thm:main} relies on the Hardy-Littlewood
circle method, implemented over the number field $K$ by Skinner
\cite{MR1446148}. We start by reviewing some notation from
\cite{MR1446148}.

Let $\places, \archplaces, \nonarchplaces$ denote the sets of all
places, archimedean places, and non-archimedean places of $K$, and
write $K_v$ for the completion of $K$ at the place $v$. If $v \in
\archplaces$ then we identify $K_v$ with the field $\RR$ or $\CC$ in
the usual way.

We identify $V=K\otimes_\QQ \RR$ with $\prod_{v \in
  \archplaces}K_v$. This allows us to naturally define the conjugates
$x^{(v)}\in K_v$ of $x \in V$ via projection and to extend the norm
and trace of $K$ to functions $N : V \to \RR$, $\Tr : V \to \RR$. On
$V$, we moreover have an $\RR$-vector norm defined by
\begin{equation*}
  \abs{x} := \max\{\abs{x_1}, \ldots, \abs{x_n}\} \quad\text{ for }\quad x = x_1\omega_1 + \cdots + x_n\omega_n
\end{equation*}
that satisfies $\abs{x} \asymp \max_{v\in\archplaces}\{\abs{x}_v\}$,
where $\abs{x}_v$ is the usual absolute value on $K_v \in \{\RR,
\CC\}$. We extend the norm
 to $V^s$ via $\abs{\xx} :=
\max_{j=1,\ldots,s}\{\abs{x_j}\}$ for $\xx = (x_1, \ldots, x_s)$.

Let
\begin{equation*}
  R := \{x_1\omega_1 + \cdots + x_n\omega_n \where x_i \in [0,1)\} \subset V.
\end{equation*}
We normalize the Haar measure on $V$ by $\vol(R) = 1$. Elements of
$V^T = \prod_{d=1}^DV^{t_d}$ are written with double indices $\va =
(\alpha_{d,i})_{\substack{1\leq d \leq D\\1\leq i \leq t_d}}$. We write
$e(x) = e^{2\pi i x}$ for $x \in \RR$ and $\Phi(x) = e(\Tr(x))$ for $x
\in V$. The circle method is based on the identity
\begin{equation}\label{eq:circle_method}
  N(P)=\int_{\va\in\R^T}S(\va)\dd\va,
\end{equation}
where
\begin{equation}
  \label{eq:def_S_alpha}
  S(\va) := \sum_{\xx \in \n^s\cap P\B}\Phi\left(\sum_{d=1}^D\sum_{i=1}^{t_D}\alpha_{d,i}G_{d,i}(\xx)\right).
\end{equation}
We divide $\R^T$ into major and minor arcs as follows. Let $\varpi \in
(0,1/3)$ be a fixed constant to be specified in Section
\ref{sec:majorarcs}. For $\gamma\in \K$, we have the denominator ideal
$\aaa_\gamma:=\{\beta\in\Ok\where\beta\gamma\in\n\}$.  For
$\vg=(\gamma_{d,i})_{d,i}\in(\R\cap K)^T$, let
$\aaa_{\vg}:=\bigcap_{d,i}\aaa_{\gamma_{d,i}}$. The major arc
corresponding to $\vg$ is
\begin{equation*}
  \M_\vg := \{\va\in\R^T\where \abs{\alpha_{d,i}-\gamma_{d,i}}\leq P^{-d+\varpi} \text{ for all }1\leq d\leq D, 1\leq i\leq t_d\},
\end{equation*}
where the distance $\abs{\alpha_{d,i}-\gamma_{d,i}}$ is to be
understood modulo $\n$. We define the major arcs
\begin{equation*}
  \M := \bigcup_{\substack{\vg\in(\R\cap K)^T\\\N\aaa_\vg\leq P^{\varpi}}}\M_\vg
\end{equation*}
and the minor arcs
\begin{equation*}
  \m := \R^T\smallsetminus \M.
\end{equation*}
In Section \ref{sec:minorarcs}, we show that, under the
hypotheses of Theorem \ref{thm:main}, the contribution of the minor
arcs $\m$ to the integral in \eqref{eq:circle_method} is absorbed by
the error term. In Sections \ref{sec:majorarcs} and \ref{sec:singint},
we evaluate the contribution of the major arcs $\M$ as
$\mathfrak{S}\mathfrak{J}P^{n(s-\mathcal{D})}+O(P^{n(s-\mathcal{D})-\delta})$,
with the singular series
\begin{equation*}
  \mathfrak{S} = \prod_\ppp\sum_{j=0}^\infty\frac{1}{\N\ppp^{js}}\sum_{\substack{\vg\in(R\cap K)^T\\\aaa_\vg = \ppp^j}}\sum_{\xx \in (\n/\ppp^j\n)^s}\Phi\left(\sum_{d=1}^D\sum_{i=1}^{t_d} \gamma_{d,i}G_{d,i}(\xx)\right)
\end{equation*}
and the singular integral
\begin{equation*}
  \mathfrak{J} = \int_{\vg\in V^T}\int_{\xx\in \mathcal{B}}\Phi\left(\sum_{d=1}^D\sum_{i=1}^{t_d} \gamma_{d,i}F_{d,i}(\xx)\right)\dd\xx\dd\vg.
\end{equation*}
In Sections \ref{sec:exponentialsums-new} and
\ref{sec:iterativeargument}, we prove the main tool to be used in our
estimations, an iterative Weyl-type lemma for the exponential sum
$S(\va)$ that generalizes the version from
\cite{BrowningHeathBrown2014} to number fields.

\subsection{Further notation}
It is sometimes useful to identify $V$ with $\RR^n$ via the basis
$\omega_1, \ldots, \omega_n$. Then $\xx \in V^s$ with
$x_i = x_{i,1}\omega_1 + \cdots + x_{i,n}\omega_n$ is identified with
$\vX = (x_{i,j})_{i,j}\in \RR^{ns}$. The volume on $V$ becomes the
usual Lebesgue measure on $\RR^n$, and the norm $\abs{\cdot}$ becomes
the usual max-norm on $\RR^n$, which we will also denote by
$\abs{\cdot}$. To each polynomial $G \in V[x_1, \ldots, x_s]$, we
associate the polynomial
\begin{equation*}
  G^*(\vX) := \Tr(G(\xx)) \in \RR[\vX]
\end{equation*}
and the system $G_j^*(\vX)$, $1\leq j \leq n$, defined
via
\begin{equation*}
  G_j^*(\vX) := \Tr(\omega_j G(\xx)) \in \RR[\vX].
\end{equation*}
Then any $\xx$ in $V^s$ satisfies $G_{d,i}(\xx) = 0$ for all $d,i$ if
and only if $G^*_{d,i,j}(\vX) = 0$ for all $d,i,j$, and our system of $T$
polynomials in $s$ variables over $\Ok$ is equivalent to a system of
$nT$ polynomials in $ns$ variables over $\ZZ$. In fact, the affine variety defined over $\QQ$ by the polynomials $G_{d,i,j}^*(\vX)$ is the Weil restriction of the $K$-variety defined by the polynomials $G_{d,i}(\xx)$. These identifications allow us to write $S(\va)$ as an exponential sum over $\ZZ^{ns}$:
\begin{equation*}
  S(\va) = \sum_{\vX \in \ZZ^{ns}\cap PB}e\left(\sum_{d=1}^D\sum_{i=1}^{t_d}\sum_{j=1}^n\alpha_{d,i,j}G^*_{d,i,j}(\vX)\right),
\end{equation*}
where $\alpha_{d,i}=\alpha_{d,i,1}\omega_1 + \cdots + \alpha_{d,i,n}\omega_n$.

We denote the standard basis of the free $V$-module $V^s$ by
$\vv_1,\ldots,\vv_s$, and the standard basis of $\RR^{ns}$ by
$\vE_{ij}$ ($i=1,\ldots,s$ and $j=1,\ldots,n$). By our identification,
we obtain $\vE_{ij}=\omega_j\vv_i$.

For $\beta\in\RR$, we write $\lVert\beta\rVert$ for the distance of
$\beta$ to the nearest integer.

For any form $F \in V[x_1, \ldots, x_s]$ of degree $d$, we write
$F(\xx_1 | \cdots | \xx_d)$ for the corresponding polar
$d$-multilinear form, normalized by $d!F(\xx) = F(\xx|\cdots|\xx)$.
Similarly, $F^*(\vX_1|\cdots |\vX_{d})$ is the polar $d$-multilinear
form corresponding to $F^*$. Observe that
$F^*(\vX_1|\cdots |\vX_d) = \Tr(F(\xx_1|\cdots|\xx_d))$.

\subsection{The singular integral}\label{subsec:singint}

Our main task in Section \ref{sec:singint} is to show that the integral
\begin{equation}\label{eq:singint_intro}
  \mathfrak{J}(H) := \int_{\substack{\vg\in V^T\\\abs{\vg}\leq H}}\int_{\xx\in \mathcal{B}}\Phi\left(\sum_{d=1}^D\sum_{i=1}^{t_d} \gamma_{d,i}F_{d,i}(\xx)\right)\dd\xx\dd\vg
\end{equation}
converges absolutely as $H \to \infty$, and that
\begin{equation*}
  \mathfrak{J}(H) - \mathfrak{J}\ll H^{-\delta}
\end{equation*}
for some positive $\delta$, see Lemma \ref{lem:singularintegral}. In
the case where all degrees are equal, i.e. $t_d = 0$ for all $d\neq
D$, this is done by Skinner in \cite[Lemma 9]{MR1446148}. For the
proof, Skinner suggests to think of our forms $F_{D,i}$ as the forms
$F_{D,i,j}^*$ over $\ZZ$ and to apply Schmidt's Lemma 8.1 from
\cite{MR781588}. Schmidt's lemma is a formalization of the classical
indirect approach to the singular integral, already used by Birch
\cite{MR0150129}, where the Weyl-type lemma used in the treatment of
the minor arcs is applied once more to bound the inner integral in
\eqref{eq:singint_intro}. Hence, it depends on a certain hypothesis,
called the \emph{restricted hypothesis} by Schmidt. Applied to our
forms $F_{D,i,j}^*$, it requires that at least one of the following
  alternatives hold for some $\Omega>nt_D+1$ and each $\Delta\in(0,1]$:
  Every $\va\in R^{t_D}$ satisfies
\begin{enumerate}
\item[(i)] $|S(\va)| \leq P^{ns-\Delta\Omega}$, or
\item[(ii)] there is $q\in\NN$, $q \leq P^\Delta$ with
  $\lVert q\alpha_{D,i,j}\rVert\leq P^{-D+\Delta}$ for all $1\leq i \leq
  t_D$ and $1\leq j \leq n$.
\end{enumerate}
Skinner gives no argument why this hypothesis would hold. In fact, we
were not able to deduce it from either Skinner's Weyl-type lemma
\cite[Lemma 2]{MR1446148}, or Birch's Weyl-type lemma \cite[Lemma
2.5]{MR0150129} applied to the $F_{D,i,j}^*$, without replacing the
lower bound \eqref{eq:skinner_condition} on the number of variables
$s$ by the stronger bound
\begin{equation}\label{eq:stronger_bound}
  s-B_D > t_D(n t_D + 1)(D-1)2^{D-1},
\end{equation}
and we see no reason why it should hold otherwise. Let us note that
with the stronger assumption \eqref{eq:stronger_bound} instead of
\eqref{eq:skinner_condition}, the main theorem of \cite{MR1446148}
would follow directly from the techniques of \cite{MR0150129} applied
to the $G_{D,i,j}^*$.

In Section \ref{sec:singint}, we apply genuine number field arguments
to our treatment of the singular integral, culminating in Lemma
\ref{lem:singularintegral}. When all degrees are equal to $D$, the
hypothesis \eqref{eq:degree_condition} of Lemma
\ref{lem:singularintegral} is exactly Skinner's hypothesis
\eqref{eq:skinner_condition}, so we prove \cite[Lemma 9]{MR1446148}
as a special case of Lemma \ref{lem:singularintegral}.

\section{Exponential sums}\label{sec:exponentialsums-new}
For a function $f : V^s \to \RR$ and $\vh \in V^s$, we write
$\Delta_\vh(f)(\xx) := f(\xx+\vh)-f(\xx)$ for the usual forward
difference operator. The following lemma is a number field analog of
van der Corput's variant of Weyl differencing.

\begin{lemma}\label{lem:differencing}
  Let $q \in \Ok\setminus\{0\}$ and $H \in \ZZ$ with $1 \leq H
  \ll P/\abs{q}$. Let $f : V^s \to \RR$ and $\mathcal{I}$ be a box
  aligned to the basis $\omega_1, \ldots, \omega_n$.  Then
  \begin{equation*}
    \abs{\sum_{\xx\in\n^s\cap P\mathcal{I}}\Phi(f(\xx))}^2 \ll
    \left(\frac{P}{H}\right)^{ns}\sum_{\substack{\vh\in\n^s\\\abs{\vh}<H}}\abs{\sum_{\xx\in
        \n^s\cap P\mathcal{I}'(\vh)}\Phi(\Delta_{q\vh}(f)(\xx))}, 
  \end{equation*}
  where $\mathcal{I}'(\vh)$ is a box aligned to the basis
  $\omega_1, \ldots, \omega_n$ that depends on $\vh$. The implicit
  constant depends only on $K$ and $s$.
\end{lemma}
 
\begin{proof}
  The proof is analogous to the version over $\QQ$ (see the proof of
  \cite[Lemma 4.1]{BrowningHeathBrown2014}). Let $\chi_{P\mathcal{I}}$
  be the indicator function of $P\mathcal{I}$. Let $R^*\subset
  V$ be the set of all $\vu = u_1\omega_1 + \cdots + u_n\omega_n \in
  V$ with $1 \leq u_j \leq H$ for all $1\leq j \leq n$. Then
  \begin{align*}
    H^{ns}\sum_{\mathbf{x}\in \n^s\cap P\mathcal{I}}\Phi(f(\xx))
    &=\sum_{\substack{\vu\in(\n\cap R^*)^s}}\sum_{\mathbf{x}\in\n^s}\Phi(f(\xx+q\vu))\chi_{P\mathcal{I}}(\xx+q\vu)\\
    &=\sum_{\substack{\xx\in\n^s\\\abs{\xx}\ll
        P}}\sum_{\substack{\vu\in(\n\cap R^*)^s}}
    \Phi(f(\xx+q\vu)\chi_{P\mathcal{I}}(\xx+q\vu).
  \end{align*}
  Here, we used that $\chi_{P\mathcal{I}}(\xx+q\vu)\neq 0$ implies
  $\abs{\xx}\leq P + \abs{q\vu}\ll P + \abs{q}\abs{\vu}\ll P$. By
  Cauchy's inequality,
  \begin{align*}
    H^{2ns}\abs{\sum_{\mathbf{x}\in \n^s\cap
      P\mathcal{I}}\Phi(f(\xx))}^2 &\ll
    P^{ns}\sum_{\substack{\xx\in\n^s\\\abs{\mathbf{x}}\ll P}}
    \abs{\sum_{\substack{\vu\in(\n\cap R^*)^s}}
    \Phi(f(\xx+q\vu))\chi_{P\mathcal{I}}(\xx+q\vu)}^2\\
    &=P^{ns}\sum_{\substack{\vh\in\n^s\\\abs{\vh} < H}}n(\vh)\sum_{y\in\n^s}\Phi(f(\vy+q\vh))\chi_{P\mathcal{I}}(\vy+q\vh)
    \overline{\Phi(f(\vy))}\chi_{P\mathcal{I}}(\vy),
\end{align*}
where
\[n(\vh)=\#\left\{(\vu,\vv)\in(\n\cap R^*)^{2s}\where \vh=\mathbf{v}-\mathbf{u}\right\}\leq H^{ns}.\]
Thus,
\begin{align*}
  \abs{\sum_{\mathbf{x}\in \n^s\cap
    P\mathcal{I}}\Phi(f(\xx))}^2 &\ll
  \left(\frac{P}{H}\right)^{ns}\sum_{\substack{\vh\in\n^s\\\abs{\vh}<
     H}} \abs{\sum_{\vy\in\n^s\cap P\mathcal{I}'(\vh)}\Phi(f(\vy+q\vh))\overline{\Phi(f(\vh))}},
\end{align*}
where $\mathcal{I}'(\vh) \subseteq \mathcal{I}$ is a box aligned to the
basis $\omega_1,\ldots, \omega_n$, depending on $\vh$.
\end{proof}

Let $f,g\in V[x_1,\ldots,x_s]$ with $\deg f\leq d$. Suppose that
$qg=g_1+g_2$ with $q\in\n\smallsetminus\{0\}$,
$g_1\in\mathcal{O}_K[x_1,\ldots,x_s]$ and $g_2\in V[x_1,\ldots,x_s]$
such that
all coefficients $a_j$ of $g_2$ of each degree $j$ satisfy
\begin{equation}\label{eq:g2_partial_bounds}
  \abs{a_j} \ll_{j} \varphi P^{-j},
\end{equation}
for some $\varphi \geq \abs{q}$.

Let $\mathcal{B}'$ be a box aligned to the basis
$\omega_1, \ldots,\omega_n$. We are interested in the estimation of
the exponential
sum \[\Sigma:=\sum_{\mathbf{x}\in\n^s\cap P\mathcal{B}'}\Phi(f(\mathbf{x})+g(\mathbf{x})).\]
In the process, $f$ will be the dominant polynomial whereas $g$
originates from the higher exponents which are already well
approximable. We aim for an estimate of the form
\[\lvert\Sigma\rvert=P^{ns}L\]
with small $L$. Let $F$ be the homogeneous part of $f$ of degree $d$,
and recall that $(qF)^*(\vX_1|\cdots|\vX_d)$ is the $d$-multilinear
polar form corresponding to the form $(qF)^*(\vX)$.

\begin{lemma}\label{lem:exponentialsums}
For $M\geq1$, we have
\begin{gather}\label{eq:expsumsetimate}
L^{2^{d-1}}\ll P^{-(d-1)ns}\left(\varphi
  M\right)^{(d-1)ns}(\log P)^{ns}\mathcal{M},
\end{gather}
where $\mathcal{M}$ is the number of all $(\xx_1,\ldots,\xx_{d-1})\in(\n^s)^{d-1}$ satisfying
\begin{align*}
  \lvert\mathbf{\xx}_i\rvert&\leq\frac{P}{\varphi M}& \quad &\text{ for all $1\leq i <d$, and }\\
  \lVert (qF)^*(\mathbf{X}_1\vert\cdots\vert\mathbf{X}_{d-1}\vert
  \vE_{i,j})\rVert&\leq\frac1{P\varphi^{d-2}M^{d-1}}& \quad &\text{ for all $1\leq i\leq s, 1\leq j\leq n$.}
\end{align*}
\end{lemma}

\begin{proof}
  This is mostly analogous to the proof of \cite[Lemma 4.1]{BrowningHeathBrown2014}. The lemma holds trivially if $\varphi \geq
  P$. Hence, we assume that $\varphi \leq P$.

  We start by $d-2$ Weyl differencing steps, that is $d-2$ applications of Lemma \ref{lem:differencing} with $q:=1$, $H:=P$, linked by Cauchy's inequality). This yields
\begin{gather}\label{eq:weyl_diff}
L^{2^{d-2}}
\ll P^{-ns(d-1)}\sum_{\substack{\vh_1\in\n^s\\\lvert\mathbf{h}_1\rvert<P}}\cdots
\sum_{\substack{\vh_{d-2}\in\n^s\\\lvert\mathbf{h}_{d-2}\rvert<P}}
\abs{\sum_{\mathbf{x}\in \n^s\cap P\mathcal{I}}\Psi(\mathbf{x})},
\end{gather}
where
\[\Psi(\mathbf{x}):=\Phi\left(\Delta_{\mathbf{h}_1,\ldots,\mathbf{h}_{d-2}}(f+g)(\mathbf{x})\right)\]
and $\mathcal{I}\subset\mathcal{B}'$ is a box aligned to the basis $\omega_1, \ldots, \omega_n$ that depends on $\vh_1, \ldots, \vh_{d-2}$.

For the $(d-1)$-st differencing step, we choose $q$ as in the setup
before our Lemma and
\[H:=\left\lfloor\frac{P}{\varphi}\right\rfloor \geq 1.\]
By Lemma \ref{lem:differencing},
\begin{align*}
\abs{\sum_{\mathbf{x}\in \n^s\cap P\mathcal{I}}\Psi(\mathbf{x})}^2
&\ll\varphi^{ns}\sum_{\substack{\vw\in\n^s\\\lvert\mathbf{w}\rvert<H}}\abs{ \sum_{\vy\in
    \n^s\cap P\mathcal{I}'}\Psi(\mathbf{y}+q\mathbf{w})\overline{\Psi(\mathbf{y})}},
\end{align*}
where $\mathcal{I}' \subseteq \mathcal{I}$ is a box aligned to the
basis $\omega_1,\ldots, \omega_n$, depending on $\vw$. Together with
Cauchy's inequality and \eqref{eq:weyl_diff}, this yields
\begin{align*}
L^{2^{d-1}}
&\ll P^{-nsd}\varphi^{ns}
\sum_{\substack{\vh_1\in\n^s\\\lvert\mathbf{h}_1\rvert<P}}\cdots
\sum_{\substack{\vh_{d-2}\in\n^s\\\lvert\mathbf{h}_{d-2}\rvert<P}}\sum_{\substack{\vw\in\n^s\\\lvert\mathbf{w}\rvert<H}}
  \abs{\sum_{\vy\in
    \n^s\cap P\mathcal{I}'}\Psi(\mathbf{y}+q\mathbf{w})\overline{\Psi(\mathbf{y})}}.
\end{align*}
Note that $\Psi$ implicitly depends on
$\mathbf{h}_1,\ldots,\mathbf{h}_{d-2}$, and $\mathcal{I}'$ depends on $\vh_1,
\ldots, \vh_{d-2}$ and $\vw$.

Now we take a closer look at the summands of the innermost sum:
\[\Psi(\mathbf{y}+q\mathbf{w})\overline{\Psi(\mathbf{y})}
=\Phi(\Delta_{\mathbf{h}_1,\ldots,\mathbf{h}_{d-2},q\mathbf{w}}(f+g)(\mathbf{y})).\]
Recall that $F$ is the leading form of $f$ of degree $d$.  Then by
linearity of $F(\xx_1\vert\cdots\vert\xx_d)$ we have that
\begin{align*}
\Delta_{\mathbf{h}_1,\ldots,\mathbf{h}_{d-2},q\mathbf{w}}(f)(\mathbf{y})
&=qF(\mathbf{h}_1\vert\cdots\vert\mathbf{h}_{d-2}\vert
\mathbf{w}\vert\mathbf{y})+C.
\end{align*}
Now we focus on $g$. Since all coefficients of $\Delta_{\vh_1, \ldots, \vh_{d-2},q\vw}(g_1)$ are $\Ok$-multiples of $q$, we have
\begin{align*}
\Phi(\Delta_{\mathbf{h}_1,\ldots,\mathbf{h}_{d-2},q\mathbf{w}}(q^{-1}g_1)(\mathbf{y}))
=1
\end{align*}
for all $\vy\in\Ok$. Since $\abs{\vh_i}\leq P$ for all
$1\leq i \leq d-2$, it follows from \eqref{eq:g2_partial_bounds} that
each coefficient $b_j$ of $\Delta_{\vh_1,\ldots,\vh_{d-2}}(g_2)$
of any degree $j$ satisfies
\begin{equation*}
  \abs{b_j}\ll_j \abs{\vh_1}\cdots\abs{\vh_{d-2}}\varphi P^{-j-(d-2)}.
\end{equation*}
Since $\abs{q\vw}\ll \abs{q} H\leq \varphi H \ll P$, the coefficients $c_j$ of
$\Delta_{\vh_1, \ldots,\vh_{d-2}, \q\vw}(q^{-1}g_2)$ of degree $j$ are
bounded by
\begin{equation*}
 \abs{c_j} \ll \abs{q^{-1}b_{j+1}q w_\ell}  \ll\abs{b_{j+1}}\abs{\vw} \ll_j\abs{\vh_1}\cdots\abs{\vh_{d-2}}\varphi P^{-j-(d-1)} H \ll P^{-j},
\end{equation*}
where $b_{j+1}$ a coefficient of
$\Delta_{\vh_1,\ldots,\vh_{d-2}}(g_2)$ of degree $j+1$ and $w_\ell$ is
a component on $\vw$.

 Write $\vu := q\vw$. With the tuples $\vH_i, \vU \in \ZZ^{sn}$ corresponding to $\vh_i, \vu\in\n^s$,  we have shown that any $j$-th order partial derivative of $\Delta_{\vH_1, \ldots, \vH_{d-2}, \vU}((q^{-1}g_2)^*)$ is $\ll_j P^{-j}$ uniformly on $[-P,P]^{ns}$.

Let $I'\subseteq [-1,1]^{sn}$ be the box in $\RR^{sn}$ corresponding to $\mathcal{\mathcal{I}}'$. Using the above computations, 
\begin{multline*}
  \sum_{\vy\in
    \n^s\cap P\mathcal{I}'}\Psi(\mathbf{y}+q\mathbf{w})\overline{\Psi(\mathbf{y})}\\ = \sum_{\vY\in\ZZ^{ns}\cap PI'}e((qF)^*(\vH_1|\cdots |\vH_{d-2}|\vW|\vY) + \Tr(C) + \Delta_{\vH_1, \ldots, \vH_{d-2}, \vU}((q^{-1}g_2)^*)(\vY)).
\end{multline*}

In the same manner as Browning and Heath-Brown, we apply
multidimensional partial summation and our uniform bounds for the
partial derivatives of
$\Delta_{\vH_1,\ldots,\vH_{d-2},\vU}((q^{-1}g_2)^*)$ to obtain
\[\sum_{\mathbf{y}\in
  \n^s\cap\mathcal{I}'}\Psi(\mathbf{y}+q\mathbf{w})\overline{\Psi(\mathbf{y})}
\ll\abs{\sum_{\mathbf{Y}\in \ZZ^{ns}\cap
I''}e\left((qF)^*(\vH_1|\ldots|\vH_{d-2}|\vW|\vY)\right)},
\]
with a box $I''\subseteq I'$ aligned to the coordinate axes.
We are now in exactly the same situation as in the proof of
\cite[Lemma 4.1]{BrowningHeathBrown2014}, just in Dimension $sn$
instead of $n$ and with $\varphi$ instead of $q\varphi$.  What remains of the proof is identical to the arguments of \cite{BrowningHeathBrown2014} starting at (4.5), just with tuples $\vY, \vH_j, \vW$ of $sn$ variables instead of tuples $\vy, \xx_j, \vw$ of $n$ variables.
\end{proof}

\section{The iterative argument}\label{sec:iterativeargument}
Our aim in this section is to find a Weyl-type estimate for the
exponential sum $S(\alpha)$ defined in~\eqref{eq:def_S_alpha}. To this end, we write
\begin{equation*}
  |S(\alpha)| = P^{ns}L.
\end{equation*}
We define $Q_{D+1}:= 1$ and, for $d\in\Delta$,
\begin{equation}\label{eq:Q_j_def}
  Q_d := (\log P)^{e(d)}L^{-s_d/n},
\end{equation}
where $e(d)$ is an explicit but irrelevant exponent which could be
computed from the arguments in the proof of Lemma
\ref{lem:iterativeargument}. For those $1\leq d\leq D$ with
$d\not\in\Delta$ we set $Q_d:=Q_k$, where $k$ is the smallest integer
bigger than $d$ in $\Delta$. Similarly, we can extend the definition of
the exponents $e(d)$ to these values.

For $j \in \Delta$, we consider upper
bounds
\begin{equation}
  \label{eq:L_bound_j}
  L^{2^{j-1}}\leq \left(\frac{Q_{j+1}}{P}\right)^{n(s-B_j)}(\log P)^{ns+1}.
\end{equation}
Let $I_d^{(1)}$ be the set of all $\va \in \R^T$ such that
\eqref{eq:L_bound_j} holds for $j = d$ but fails for every $j >
d$.
Moreover, let $I^{(2)}$ be the set of all $\va \in \R^T$ such that
\eqref{eq:L_bound_j} fails for all $j \in \Delta$. We are going to
prove the following number field analogue of \cite[Lemma 6.2]{BrowningHeathBrown2014}.

\begin{lemma}\label{lem:iterativeargument}
  Let $d \in \Delta$ and $P \gg 1$. If $\va \in I_d^{(1)}$ then
  \begin{equation}\label{iterativeestimate}
    L^{2^{d-1}+(s-B_d)s_{d+1}}\ll P^{-n(s-B_d)+\epsilon}.
  \end{equation}
  Moreover, there are $\q_j \in \n$, $\vn_j \in \n^{t_j}$ for all $d <
  j \leq D, j \in \Delta$ satisfying
  \begin{align}
    q_k &\mid \q_j&& &\text{ for all }k > j, k \in \Delta\label{eq:q_j_divisibility}\\
    \abs{\q_j}&\leq Q_j&& &\label{eq:q_j_bound}\\
    \abs{\q_j\alpha_{j,i} - \nu_{j,i}}&\leq Q_jP^{-j}&& &\text{ for all }
    1\leq i\leq t_j.\label{eq:q_j_approx}
  \end{align}
  If $\va \in I^{(2)}$ then there are $\q\in \n$,
  $\vn_j\in\n^{t_j}$ for all $j \in \Delta$, satisfying
  \eqref{eq:q_j_divisibility}, \eqref{eq:q_j_bound}, and
  \eqref{eq:q_j_approx}.
\end{lemma}

The idea is to iteratively apply Lemma
\ref{lem:exponentialsums}. Recall that $\vv_j$ denotes the $j$-th
element of the standard basis of $V^s$. For $d\leq D$, we define the
matrix
\begin{equation*}
  \widehat{J}_d(\xx_1,\ldots,\xx_{d-1}) := \left(F_{d,i}(\xx_1|\cdots|\xx_{d-1}|\vv_j)\right)_{\substack{1\leq i\leq t_d\\1\leq j\leq s}}
\end{equation*}
and the corresponding affine variety
$\widehat{S}_d \subseteq (\AA_K^{s})^{d-1}$ defined by the condition
\begin{equation*}
\rank(\widehat{J}_d(\xx_1,\ldots,\xx_{d-1}))<t_d.
\end{equation*}
We need an estimate for the number of integral points on
$\widehat{S}_d$ of bounded norm. Let
\begin{equation*}
  \mathcal{M}_0(P) := \#\{(\xx_1, \ldots, \xx_{d-1})\in \widehat{S}_d(K)\cap (\n^s)^{d-1}\where \abs{\xx_i}\leq P \text{ for all } 1\leq i <d\}.  
\end{equation*}

\begin{lemma}\label{lem:ptsonaffinesubvar}
  For $P \geq 1$, we have
  \[\mathcal{M}_0(P)\ll P^{n(B_d+s(d-2))}.\]
\end{lemma}

\begin{proof}
  As in \cite[Lemma 5.1]{BrowningHeathBrown2014}, using \cite[Lemma
  3]{MR1446148} instead of \cite[Lemma 3.1]{MR0150129}.
\end{proof}

The main tool for the proof of Lemma \ref{lem:iterativeargument} is the
following iterative argument.
 
\begin{lemma}\label{lem:iterationlemma}
Let $\abs{S(\va)}=P^{ns}L$ and $d\in\Delta$. Furthermore suppose that
\begin{itemize}
  \item either $d=D$, $q=1$ and $Q=1$, 
  \item or $d<D$, and there exist $Q\geq 1$ and $q\in\n$ with
    $\abs{q}\leq Q$, and $\vn_{j}\in\n^{t_j}$, such that
\[\abs{q\alpha_{j,i}-\nu_{j,i}}\leq QP^{-j}\quad\text{for }d<j\leq D\text{
  and }1\leq i\leq t_j.\]
\end{itemize}
Then, for $P$ sufficiently large, either
\[L^{2^{d-1}}\leq\left(\frac QP\right)^{n(s-B_d)}(\log P)^{ns+1},\]
or there exists $q^*\in\n$ with
\[\abs{q^*}\leq Q^*
:=\left(\frac{(\log
    P)^{ns+1}}{L^{2^{d-1}}}\right)^{\frac{t_d(d-1)}{n(s-B_d)}}(\log
P)\]
and $\vn_{d}\in\n^{t_d}$, such that
\[\abs{q^*q\alpha_{d,i}-\nu_{d,i}}\leq QQ^*P^{-d}\text{ for }1\leq i\leq t_d.\]
\end{lemma}

\begin{proof}
  The key tool is Lemma \ref{lem:exponentialsums}. We distinguish the
  two cases $d=D$ and $d<D$:
  \begin{itemize}
  \item $\mathbf{d=D}$: In this case, we choose $\varphi:=1$,
    $g=g_1=g_2:=0$, and
    \[f(\xx):=\sum_{j=1}^D\sum_{i=1}^{t_j}\alpha_{j,i}G_{j,i}(\xx)\]
  \item $\mathbf{d<D}$: Then we let
    \begin{align*}
f(\xx):=\sum_{j=1}^d\sum_{i=1}^{t_j}\alpha_{j,i}G_{j,i}(\xx)\quad\text{and}\quad
g(\xx):=\sum_{j=d+1}^D\sum_{i=1}^{t_j}\alpha_{j,i}G_{j,i}(\xx).
    \end{align*}
    By hypothesis, we have $q\alpha_{j,i}=\nu_{j,i}+\theta_{j,i}$ for
    $d<j\leq D$ and $1\leq i\leq t_j$, with $\abs{\theta_{j,i}}\leq
    QP^{-j}$.
    We write $qg = g_1 +g_2$, where
    \begin{align*}
      g_1 := \sum_{j=d+1}^D\sum_{i=1}^{t_j}\nu_{j,i}G_{j,i}(\xx)\quad\text{ and }\quad
      g_2 := \sum_{j=d+1}^D\sum_{i=1}^{t_j}\theta_{j,i}G_{j,i}(\xx).
    \end{align*}
    This allows us to choose $\varphi := Q$.
  \end{itemize}
  
  Now we apply Lemma \ref{lem:exponentialsums} with $M:=\max\{1,M_1\}$, where
  \[M_1:=\frac P Q\left(\frac{L^{2^{d-1}}}{(\log P)^{ns+1}}\right)^{\frac1{n(s-B_d)}}.\]
  If $M_1\leq1$ then
  \[\frac{L^{2^{d-1}}}{(\log P)^{ns+1}}\leq\left(\frac
    QP\right)^{n(s-B_d)},\]
  as required by the first alternative in the conclusion of the lemma.

  Therefore, we may suppose that $M=M_1>1$ and consider two cases
  according to whether all points
  $(\xx_1, \ldots, \xx_{d-1})\in(\n^{s})^{d-1}$ counted by $\mathcal{M}$
  from Lemma \ref{lem:exponentialsums} are in the affine variety
  $\widehat{S}_d$ or not.

  If they all lie in $\widehat{S}_d$ then an application of Lemma
  \ref{lem:ptsonaffinesubvar} implies
  \[\mathcal{M}\leq\mathcal{M}_0\left(\frac{P}{QM}\right)
  \ll\left(\frac{P}{QM}\right)^{nB_d+ns(d-2)}.\]
  Hence we have
  \begin{align*}
  L^{2^{d-1}}&\ll P^{-(d-1)sn}(QM)^{(d-1)ns}(\log
  P)^{ns}\left(\frac{P}{QM}\right)^{nB_d+ns(d-2)}\\
  &\ll(QM)^{ns-nB_d}P^{nB_d-ns}(\log P)^{ns}\\
  &=\left(\frac{QM}{P}\right)^{n(s-B_d)}(\log P)^{ns}.
  \end{align*}
  Substituting $M$ yields
  \[L^{2^{d-1}}\ll L^{2^{d-1}}(\log P)^{-1}\]
  which is a contradiction for $P$ sufficiently large.

  In the remaining case, we are given a  point
  $(\xx_1,\ldots,\xx_{d-1})\in(\n^s)^{d-1}$ with
  \begin{itemize}
  \item $\abs{\xx_i}\leq\frac{P}{QM}$ for all $1\leq i <d$,
  \item
    $\lVert
    (qF)^*(\vX_1|\cdots|\vX_{d-1}|\vE_{p,\ell})\rVert\leq\frac1{PQ^{d-2}M^{d-1}}$
    for all $1\leq p\leq s$, $1\leq \ell\leq n$, and
  \item $\rank(\widehat{J}_d(\xx_1,\ldots,\xx_{d-1}))=t_d$.
  \end{itemize}

  Without loss of generality, we assume that the matrix $W$ consisting
  of the first $t_d$ columns of $\widehat{J}_d(\xx_1,\ldots,\xx_{d-1})$
  has full rank. Let $\widetilde{q}^*:=\det W$. Then
  $\widetilde{q}^*\in\n$ and
\[\abs{\widetilde{q}^*}\ll\abs{\xx_1}^{t_d}\cdots\abs{\xx_{d-1}}^{t_d}\ll\left(\frac{P}{QM}\right)^{t_d(d-1)}.\]
  We set
\[\beta_p^*=\sum_{k=1}^n\beta_{p,k}^*\omega_k:=qF(\xx_1|\xx_2|\cdots|\xx_{d-1}|\vv_p)
  =\sum_{i=1}^{t_d}\alpha_{d,i}qF_{d,i}(\xx_1|\cdots|\xx_{d-1}|\vv_p).\]
  Then
  \[\Tr(\omega_\ell\beta_p^*)=\Tr\left(\omega_\ell  qF(\xx_1|\cdots|\xx_{d-1}|\vv_p)\right)=(qF)^*(\vX_1|\cdots|\vX_{d-1}|\vE_{p,\ell}),\]
  so
  \[\lVert\Tr(\omega_\ell\beta_p^*)\rVert\leq\frac1{PQ^{d-2}M^{d-1}}.\]
  Thus, we can write
  \begin{equation}\label{eq:linear_system}
  \sum_{k=1}^n\beta_{p,k}^*\Tr(\omega_\ell\omega_k)=\Tr(\omega_\ell\beta_p^*)
  = a_{p,\ell} + d_{p,\ell},
  \end{equation}
  with $a_{p,\ell}\in\ZZ$ and
  $\abs{d_{p,\ell}}\leq (PQ^{d-2}M^{d-1})^{-1}$. With
  \begin{align*}\Omega&:=\left(\Tr(\omega_\ell\omega_k)\right)_{k,\ell=1,\ldots,n}  &
  B_p^*&:=(\beta_{p,1}^*,\ldots,\beta_{p,n}^*)\\
  A_p&:=(a_{p,1},\ldots,a_{p,n}) &
  D_p&:=(d_{p,1},\ldots,d_{p,n}),
  \end{align*}
  we can write \eqref{eq:linear_system} as
  \[B_p^*\Omega=A_p+D_p.\]
  Therefore,
  \[B_p^*=A_p\Omega^{-1}+D_p\Omega^{-1}=:A_p'+D_p',\]
  Write $d_p' := d_{p,1}'\omega_1 + \cdots + d_{p,n}'\omega_n$, where
  $D_p' = (d_{p,1}', \ldots, d_{p,n}')$, and define $a_p'$
  analogously. Then
  \[\abs{d_p'}\ll\max_k\{\abs{d_{p,k}}\}\leq\frac1{PQ^{d-2}M^{d-1}}.\]
  Let $\va_{d} := (\alpha_{d,1}, \ldots, \alpha_{d,t_d})$. On the
  one hand, by our definition of $\beta_p^*$ we see that
  \[\va_d \cdot qW=\left(\beta_p^*\right)_{1\leq p\leq t_d}.\]
  On the other hand, since $W$ has full rank there exists $\vn_d\in\n^{t_d}$ such that
  \[\vn_d\cdot W=\det(\Omega)\widetilde{q}^*(a'_p)_{1\leq p\leq t_d}.\]
  Subtracting one from the other yields
\[(\det(\Omega)\widetilde{q}^*q\va_d-\vn_d)\cdot W=\det(\Omega)\widetilde{q}^*(d_p')_{1\leq
    p\leq t_d}.\]
  We let $q^*:=\widetilde{q}^*\det(\Omega)$ and obtain
  \begin{align*}
  q^*q\va_d-\vn_d&=q^*(d_p')_{1\leq p\leq t_d}W^{-1}\\
  &\ll\abs{\xx_1}^{t_d-1}\cdots\abs{\xx_{d-1}}^{t_d-1}\max_p\{\abs{d_p'}\}\\
  &\ll\left(\frac{P}{QM}\right)^{(t_d-1)(d-1)}\frac1{PQ^{d-2}M^{d-1}}.
  \end{align*}
  Furthermore, we have
  \[\abs{q^*}\ll\abs{\widetilde{q}^*}\ll\left(\frac{P}{QM}\right)^{t_d(d-1)},\]
  and thus
  \[\abs{q^*}\leq Q^*=\left(\frac{P}{QM}\right)^{t_d(d-1)}(\log P)\]
  for large enough $P$.
\end{proof}

Now we are ready to prove Lemma \ref{lem:iterativeargument}.
\begin{proof}[Proof of Lemma \ref{lem:iterativeargument}]
  We iteratively apply the preceding lemma in order to
  reduce the degree of $f$ in every step. In the first step with
  $d=D$, we see that either
  \[L^{2^{D-1}}\leq P^{n(B_d-s)}(\log P)^{ns+1},\]
  and hence $\va\in I_D^{(1)}$, or there is a $q_D\leq Q_D$, with
\[Q_D=\left((\log
  P)^{ns+1}L^{-2^{D-1}}\right)^{\frac{(D-1)t_D}{n(s-B_D)}}\log P,\]
and $\vn_D\in\n^{t_D}$ such that
\[\abs{q\alpha_{D,i}-\nu_{D,i}}\ll QP^{-D}\quad(1\leq i\leq t_D).\]
In the second case, then we apply Lemma \ref{lem:iterationlemma}
with $d:=\max\left\{\Delta\setminus\{D\}\right\}$. Then either
\[L^{2^{d-1}}\leq \left(\frac{Q_D}{P}\right)^{n(s-B_d)}(\log
P)^{ns+1},\]
and thus $\va\in I_d^{(1)}$, or there is a $q_d:=q_Dq^*\leq Q_d:=Q_DQ^*$ with
\[Q^*=\left(\frac{(\log
    P)^{ns+1}}{L^{2^{d-1}}}\right)^{\frac{t_d(d-1)}{n(s-B_d)}}\log
  P,\]
and $\vn_d\in\n^{t_d}$, such that
\[\abs{q_d\alpha_{d,i}-\nu_{d,i}}\leq Q_dP^{-d}\quad(1\leq i\leq t_d).\]
Since we also have
\begin{equation*}
  \abs{q_d\alpha_{D,i}-q^*\nu_{D,i}}\leq Q^*Q_DP^{-D}=Q_dP^{-D},  
\end{equation*}
so we may apply Lemma \ref{lem:iterationlemma} again with the next
lower value of $d$. Iterating this process we get sequences of $q_d$
and $Q_d$ for decreasing values of $d\in\Delta$. The set of $\va$ such
that for all $d\in\Delta$ the second case of Lemma
\ref{lem:iterationlemma} holds is exactly $I^{(2)}$.
\end{proof}

\section{Minor arcs}\label{sec:minorarcs}
First, let us consider the integral of $S(\va)$ over $I_D^{(1)}$.

\begin{lemma}\label{lem:minarcs_D}
  If
  \begin{equation}\label{eq:var_cond_D}
    \mathcal{D}\frac{2^{D-1}}{s-B_D}<1,
  \end{equation}
  then
  \begin{equation*}
    \int_{I_D^{(1)}}\abs{S(\va)}\dd\va \ll P^{n(s-\mathcal{D})-\delta},
  \end{equation*}
  for some $\delta>0$.
\end{lemma}

\begin{proof}
  For $\va \in I_D^{(1)}$, we have
  \[L^{2^{D-1}}\leq P^{-n(s-B_D)}(\log P)^{ns+1} \leq
  P^{-n(s-B_D)+\varepsilon}.\] 

  Therefore the integral can be estimated by
\begin{align*}
  \int_{I_D^{(1)}}\abs{S(\va)}\dd\va &\ll\vol(I_D^{(1)})\sup_{\va\in I_D^{(1)}}\abs{S(\va)}\ll 1\cdot P^{ns} P^{\frac{-n(s-B_D)}{2^{D-1}}+\epsilon}\\ &=
  P^{n\left(s-\frac{s-B_D}{2^{D-1}}\right)+\varepsilon}\ll  P^{n(s-\mathcal{D})-\delta}.
\end{align*}
for a suitable $\delta>0$, using \eqref{eq:var_cond_D}, provided that
$\epsilon$ was chosen small enough.
\end{proof}

We split $R^T$
into dyadic sets as follows.  For any $L_0>0$, let
\[\mathcal{A}(L_0):=\left\{\va\in R^T\colon
\abs{S(\va)}=P^{ns}L\text{ with }L_0<L\leq 2L_0\right\}.\]
For $I = I_d^{(1)}$, $d<D$, or $I = I^{(2)}$, we write $\mathcal{A}(L_0;I):=I\cap\mathcal{A}(L_0)\cap \m$ and estimate the integral
\[T(L_0;I):=\int_{\mathcal{A}(L_0;I)}\abs{S(\va)}\dd\va.\]
We will make use of the following facts.

\begin{lemma}\label{lem:count_divisors}
  Let $\epsilon>0$ and $a\in\OO_K$ with $N(a)\leq H$. Then the number of $b \in \OO_K$ with $b \mid a$ and $\abs{b}\leq H$ is $\ll_\epsilon H^\epsilon$. 
\end{lemma}

\begin{proof}
  There are at most $\ll_\epsilon H^{\epsilon/2}$ ideals of $\OO_K$ dividing the principal ideal $a\OO_K$. Let $\bbb$ be any principal ideal among these divisors. The number of generators of $\bbb$ with all conjugates bounded by $\ll H$ is $\ll_\epsilon H^{\epsilon/2}$, which one can see by counting units with bounded conjugates (for example, as in the proof of \cite[Lemma 7.2]{freipieropan2014}).
\end{proof}

\begin{lemma}\label{lem:L_bounded_below}
  There are positive constants $e_0, e_1$ such that all $\va \in R^T \smallsetminus I_D^{(1)}$ satisfy $L \gg P^{-e_0}$ and $Q_j \ll P^{e_1}$ for all $j$.
\end{lemma}

\begin{proof}
  The lower bound for $L$ follows directly from the definition of
  $I_D^{(1)}$. The upper bound for $Q_j$ is an immediate consequence
  of this.
\end{proof}

\begin{lemma}\label{lem:minarcs_d}
  Let $d \in \Delta$, $d<D$. If
  \begin{equation}\label{eq:var_cond_d}
    \mathcal{D}_d\left(\frac{2^{d-1}}{s-B_d}+s_{d+1}\right)+s_{d+1}\sum_{j=d+1}^Ds_jt_j<1,
  \end{equation}
  then $T(L_0;I_d^{(1)}) \ll P^{n(s-\mathcal{D})-\delta}$ for some $\delta>0$.
 \end{lemma}

\begin{proof}
  By Lemma \ref{lem:iterativeargument}, every $\va \in I_d^{(1)}$ satisfies
  \[L^{\frac{2^{d-1}}{n(s-B_d)}+\frac{s_{d+1}}{n}}\ll
  P^{-1+\varepsilon},\] and there are $q_j\in\n\smallsetminus\{0\},
  \vn_j\in\n^{t_j}$ for all $j\in\Delta$, $j>d$, satisfying
  \eqref{eq:q_j_divisibility}, \eqref{eq:q_j_bound}, and
  \eqref{eq:q_j_approx}. 

Since $q_jR$ is a fundamental domain for the ideal lattice $q_j\n\subseteq \n$ in $V$, there are exactly $\abs{N(q_j)}$ points of $\n$ in $q_jR$. Hence, for any given $q_j$, it is enough to consider $\ll\abs{N(q_j)}^{t_j}$ elements $\vn_j$.

Let us estimate the volume of the set of all $(\alpha_{j,i})_{j>d,1\leq i\leq t_d}$ belonging to a given $q_j, \vn_j$. By \eqref{eq:q_j_approx}, we see that every coordinate $q_j\alpha_{j,i}$ takes values in a set of volume $\leq Q_j^nP^{-jn}$. Since multiplication by $q_j$ is an $\RR$-linear transformation on $V$ of determinant $\asymp N(q_j)$, each $\alpha_{i,j}$ belongs to a set of volume $\ll N(q_j)^{-1}Q_j^nP^{-jn}$, and hence the total volume is $\ll \prod_{j=d+1}^D(\abs{N(q_j)}^{-1}Q_j^n P^{-jn})^{t_j}$. Due to \eqref{eq:q_j_divisibility}, Lemma \ref{lem:count_divisors} and Lemma \ref{lem:L_bounded_below}, each choice of $q_{d+1}$ defines $\ll_\epsilon P^\epsilon$ values of $q_{d+2}, \ldots, q_D$. Summing over all these $q_j$ and all the corresponding $\vn_j$, we see that
\[\vol \mathcal{A}(L_0;I_d^{(1)})\ll P^\varepsilon Q_{d+1}^n\prod_{j=d+1}^D(Q_jP^{-j})^{nt_j}\ll
P^{2\varepsilon-n\sum_{j=d+1}^Djt_j}L_0^{-s_{d+1}-\sum_{j=d+1}^Ds_jt_j}.\]
Therefore,
\begin{align*}
T(L_0;I_d^{(1)})\ll
P^{n(s-\mathcal{D}+\mathcal{D}_d)+2\varepsilon}L_0^{1-s_{d+1}-\sum_{j=d+1}^Ds_jt_j}\ll P^{n(s-\mathcal{D})-\delta},
\end{align*}
as long as \eqref{eq:var_cond_d} holds and $\epsilon$ is small enough.
\end{proof}

Finally, we concentrate on the integral over
$\mathcal{A}(L_0;I^{(2)})$. In particular, we will make use of the
fact that $\mathcal{A}(L_0;I^{(2)})\subseteq\m$.

\begin{lemma}\label{lem:minarcs_2}
  Let $d \in \Delta$, $d<D$. If
  \begin{equation}\label{eq:var_cond_0}
    s_1+\sum_{j=1}^Ds_jt_j<1.
  \end{equation}
  then $T(L_0;I^{(2)}) \ll P^{n(s-\mathcal{D})-\delta}$ for some $\delta>0$.
 \end{lemma}

 \begin{proof}
   For each $\va\in I^{(2)}$, we have $q_d \in \n\smallsetminus \{0\}$
   and $\vn_d \in \n^{t_d}$, $d\in\Delta$, with
   \eqref{eq:q_j_divisibility}, \eqref{eq:q_j_bound} and
   \eqref{eq:q_j_approx}, and as in the previous lemma it suffices to
   consider $\ll\abs{N(q_d)}^{t_d}$ tuples $\vn_d$ for each $q_d$.

   Let $\vg := (q_d^{-1}\nu_{d,i})_{d,i}$. Then it is not hard to see
   that
   \begin{align*}
     \abs{\alpha_{d,i}-\gamma_{d,i}}&\ll Q_d^nP^{-d} \text{ for all $j,d$ and }\\
     \N\aaa_\vg &\ll Q_1^n.
   \end{align*}
   With $e_{\text{max}}:=\max_d\{e(d)\}$, we have
   \begin{equation*}
     Q_d^n \ll L^{-s_1}(\log P)^{n e_{\text{max}}}.
   \end{equation*}
   Let $\varpi$ be as in the definition of the major arcs, and suppose
   that $L \geq P^{-\varpi/(2s_{1})}$. If $P$ is large
   enough, we deduce that
   \begin{align*}
     \abs{\alpha_{d,i}-\gamma_{d,i}}&\ll P^{-d+\varpi} \text{ for all $j,d$ and }\\
     \N\aaa_\vg &\ll P^\varpi,
   \end{align*}
   and hence $\va \in \M$. We conclude that $T(L_0;I^{(2)})=0$ unless 
   \begin{equation}
     L_0 \ll P^{-\varpi/(2s_{1})}.\label{eq:L0_bound}
   \end{equation}
   Let us assume that \eqref{eq:L0_bound} holds. As in the proof of
   Lemma \ref{lem:minarcs_d}, we see that
   \[\vol(\mathcal{A}(L_0,I^{(2)}))
   \ll P^\varepsilon Q_1^n\prod_{j=1}^D(Q_jP^{-j})^{nt_j} \ll
   P^{2\varepsilon-n\sum_{j=1}^Djt_j}L_0^{-s_1-\sum_{j=1}^Ds_jt_j}.\]
   This implies that
   \[T(L_0;I^{(2)})\ll
   P^{n(s-\mathcal{D})+2\varepsilon}L_0^{1-s_1-\sum_{j=1}^Ds_jt_j}\ll
   P^{n(s-\mathcal{D})-\delta},\] provided that \eqref{eq:var_cond_0}
   holds and $\epsilon$ is small enough.
   \end{proof}

   The previous lemmata allow us to estimate the integral of
   $\abs{S(\va)}$ over $\va \in \m$. Lemma \ref{lem:minarcs_D} gives a
   sufficient bound for the integral over $\m \cap I_D^{(1)}$. For
   $\va\in I_d^{(1)}$, $d<D$, or $\va\in I^{(2)}$, we have $c_0
   P^{-e_0}\leq L \leq c_1$, with constants $c_0,c_1$ independent from
   $P$. We split this interval in dyadic parts and obtain
   \begin{equation*}
     \int_{\m\cap I_d^{(1)}}\abs{S(\va)}\dd\va \ll \sum_{j = 0}^{\lceil\log_2(c_0^{-1}c_1P^{e_0})\rceil}T(2^jc_0P^{-e_0},I_d^{(1)}) \ll P^{n(s-\mathcal{D})-\delta}(\log P)
   \end{equation*}
   by Lemma \ref{lem:minarcs_d}. An analogous argument using Lemma
   \ref{lem:minarcs_2} bounds the integral over $\m \cap I^{(2)}$.

\section{Major arcs: singular series}\label{sec:majorarcs}

We now choose the parameter $\varpi$ in the definition of the major arcs by
\begin{equation*}
\varpi := \frac{1}{4+(n+1)T}.
\end{equation*}
Furthermore recall that $\mathfrak{B}\subset V^s$ is a box aligned to
the basis and $B\subseteq[-1,1]^{ns}$ the corresponding box in
$\R^{ns}$.

We start by showing that the major arcs are disjoint in pairs provided $P$ is large enough.
\begin{lemma}\label{lem:major_arcs_disjoint}
  Let $\vg_1\neq \vg_2 \in (\R\cap\K)^T$ with $\N\aaa_{\vg_j}\leq P^{\varpi}$ for $j\in\{1,2\}$. For $P\gg 1$, we have $\M_{\vg_1}\cap\M_{\vg_2} = \emptyset$.
\end{lemma}

\begin{proof}
  If $\va \in \M_{\vg_1}\cap\M_{\vg_2}$ then, writing $\vg_j =
  (\gamma_{j,d,i})_{d,i}$,
  \begin{equation*}    \abs{\gamma_{1,d,i}-\gamma_{2,d,i}}\leq\abs{\gamma_{1,d,i}-\alpha_{d,i}}+\abs{\alpha_{d,i}-\gamma_{2,d,i}}\ll P^{-d+\varpi} \leq P^{-1+\varpi}
  \end{equation*}
  holds for all $1\leq d\leq D, 1\leq i\leq t_d$. By Minkowski's
  convex body theorem, there is a nonzero $q\in\aaa_{\vg_1}\cap\aaa_{\vg_2}$ with $\abs{q}\ll P^{2\varpi/n}$. Hence, $q(\gamma_{1,d,i}-\gamma_{2,d,i})\in\n$ and $\abs{q(\gamma_{1,d,i}-\gamma_{2,d,i})}\ll P^{-1+(1+2/n)\varpi}$ for all $d,i$. Since $\varpi < 1/3$, this implies that $\vg_1=\vg_2$ whenever $P$ is large enough.
\end{proof}

For $\vg \in (\R\cap\K)^T$, we define
\begin{equation*}
\Sigma(\vg):=\sum_{\xx\in(\n/\aaa_\vg\n)^s}\Phi\left(\sum_{d=1}^D\sum_{i=1}^{t_d}\gamma_{d,i}G_{d,i}(\xx)\right),
\end{equation*}
and for $\vg \in V^T$, let
\begin{equation*}
J(\vg):=\int_{\B}\Phi\left(\sum_{d=1}^D\sum_{i=1}^{t_d}\gamma_{d,i}F_{d,i}(\xx)\right)\dd\xx.
\end{equation*}

\begin{lemma}\label{lem:major_arcs_S}
  For $\vg\in(\R\cap \K)^T$ with $\N\aaa_\vg \leq P^\varpi$, let
  $\va\in\M_\vg$ and write $\alpha_{d,i} = \gamma_{d,i} +
  \theta_{d,i}$ for all $1\leq d \leq D$ and $1\leq i \leq t_d$. Then
  \begin{equation*}
    S(\va) = \N\aaa_\vg^{-s}P^{ns}\Sigma(\vg)J((\theta_{d,i}P^d)_{d,i}) + O\left(\N\aaa_\vg\sum_{d=1}^D\sum_{i=1}^{t_d}\abs{\theta_{d,i}}P^{ns+d-1}+\N\aaa_\vg P^{ns-1}\right).
  \end{equation*}
\end{lemma}

\begin{proof}
  Whenever $d,i,j$ appear as indices of a sum, the sum runs over
  $1\leq d \leq D$, $1\leq i \leq t_d$, and $1\leq j\leq n$. As usual,
  we write $\gamma_{d,i}=\gamma_{d,i,1}\omega_1 + \cdots +
  \gamma_{d,i,n}\omega_n$, and similarly
  $\theta_{d,i}=\theta_{d,i,1}\omega_1 + \cdots +
  \theta_{d,i,n}\omega_n$. With these conventions, we have
   \begin{equation*}
     S(\va)=\sum_{\vX\in\ZZ^{sn}\cap PB}e\left(\sum_{d,i,j}(\theta_{d,i,j}+\gamma_{d,i,j})G_{d,i,j}^*(\vX)\right).
   \end{equation*}
 Let $N := \N\aaa_\vg \in(\NN\cap\aaa_\vg)$. Then $N\vg\in\n^T$, so
  in particular $N\gamma_{d,i,j}\in\ZZ$ for all $d,i,j$. Applying the standard argument over $\QQ$, we see that $S(\va)$ is the sum of
  \begin{equation*}
    \frac{1}{N^{ns}}\sum_{\vY\in([0,N-1]\cap \ZZ)^{sn}}e\left(\sum_{d,i,j}\gamma_{d,i,j}G_{d,i,j}^*(\vY)\right)\cdot J((\theta_{d,i}P^d)_{d,i})\cdot P^{ns}
  \end{equation*}
and an error term as in the lemma.
  It remains to show that
  \begin{equation*}
     \frac{1}{N^{ns}}\sum_{\vY\in([0,N-1]\cap \ZZ)^{sn}}e\left(\sum_{d,i,j}\gamma_{d,i,j}G_{d,i,j}^*(\vY)\right) = \frac{1}{N^s}\Sigma(\vg).
  \end{equation*}
 This follows from the following observations. Write $\vy = (y_1, \ldots, y_s)$, with, as usual, $y_{j}=y_{j,1}\omega_1+\cdots +y_{j,n}\omega_n$. If $\vY$ runs through $([0,N-1]\cap\ZZ)^{ns}$ then $\vy$ runs through a set of representatives of $(\n/(N\n))^s$. Moreover,
 \begin{equation*}
   e\left(\sum_{d,i,j}\gamma_{d,i,j}G_{d,i,j}^*(\vY)\right) = \Phi\left(\sum_{d,i}\gamma_{d,i}G_{d,i}(\vy)\right)
 \end{equation*}
depends only on $\vy$ modulo $\aaa_\vg\n$, and each coset of $(\n/(N\n))^s$ modulo $\aaa_\vg\n$ has $N^{(n-1)s}$ elements.
\end{proof}

For $H>0$, let
\begin{equation*}
\mathfrak{S}(H):=\sum_{\substack{\vg\in(\R\cap\K)^T\\\N\aaa_{\vg}\leq H}}\frac{\Sigma(\vg)}{\N\aaa_\vg^s}
\end{equation*}
and
\begin{equation*}
  \mathfrak{J}(H):=\int_{\substack{\vg\in V^T\\\abs{\vg}\leq H}}J(\vg)\dd\vg.
\end{equation*}

\begin{lemma}\label{lem:major_arcs_asympt}
  There is a positive constant $\delta$ such that, for large enough $P$,
  \begin{equation*}
    \int_{\M}S(\va)\dd\va = \mathfrak{S}(P^\varpi)\mathfrak{J}(P^\varpi)P^{n(s-\mathcal{D})} + O(P^{n(s-\mathcal{D})-\delta}).
  \end{equation*}
\end{lemma}

\begin{proof}
  By Lemma \ref{lem:major_arcs_disjoint}, we have
  \begin{equation*}
    \int_{\M}S(\va)\dd\va = \sum_{\substack{\vg\in(\R\cap\K)^T\\\N\aaa_\vg\leq P^{\varpi}}}\int_{\M_\vg}S(\va)\dd\va.
  \end{equation*}
  Using Lemma \ref{lem:major_arcs_S} and the obvious fact that
  $\vol(\M_\vg) = P^{-n\mathcal{D}+nT\varpi}$, it follows that
  \begin{align*}
    \int_{\M_\vg}S(\va)\dd\va &= \frac{1}{\N\aaa_\vg^s}P^{ns}\Sigma(\vg)\int_{\abs{\theta_{d,i}}\leq P^{-d+\varpi}}J((\theta_{d,i}P^d)_{d,i})\dd\vt\\ &+O(P^{n(s-\mathcal{D})-1+(nT+2)\varpi}).
  \end{align*}
  After a coordinate change in the integral over $\vt$ and summing over all $\vg$, we obtain
  \begin{equation*}
    \int_\M S(\va)\dd\va = \mathfrak{S}(P^\varpi)\mathfrak{J}(P^\varpi)P^{n(s-\mathcal{D})} + O\left(m(P^\varpi)\cdot P^{n(s-\mathcal{D})-1+(nT+2)\varpi}\right),
  \end{equation*}
  where
  \begin{equation*}
    m(P^\varpi):=|\{\vg\in (\R\cap \K)^T\mid \N\aaa_\vg \leq P^{\varpi}\}|\ll P^{\varpi (T+1)}.
  \end{equation*}
  Hence, we obtain an error term
  \begin{equation*}
    O(P^{n(s-\mathcal{D})-1+\varpi((n+1)T+3)}) = O(P^{n(s-\mathcal{D})-\delta}).\qedhere
  \end{equation*}
\end{proof}

Whenever the respective limit exists, we let
\begin{equation*}
\mathfrak{S}:=\lim_{H\to\infty}\mathfrak{S}(H) \quad\text{ and }\quad
\mathfrak{J}:=\lim_{H\to\infty}\mathfrak{J}(H).
\end{equation*}
The rest of this section is devoted to the absolute and fast convergence of the singular series $\mathfrak{S}$. The singular integral $\mathfrak{J}$ will be treated in the next section.

We start with an estimate for $\Sigma(\vg)$. Let
$\vg\in(\R\cap\K)^T$. By definition of $\aaa_\vg$, we can write
$\gamma_{d,i}\Ok=\frac{\n\aaa_{d,i}}{\aaa_\vg}$, where $\aaa_{d,i}$ is
an ideal of $\Ok$ and $\aaa_\vg +
\sum_{d=1}^D\sum_{i=1}^{t_d}\aaa_{d,i}=\Ok$.

\begin{lemma}\label{lem:major_arcs_sigma_estimate}
  Write $\gamma_{d,i}\Ok = \frac{\n\aaa_{d,i}}{\aaa_\vg}$ as
  above. Then, for $\epsilon >0$,
  \begin{equation*}
     \Sigma(\vg)\ll\N\aaa_\vg^{s+\epsilon}\min_{j\in\Delta}\left\{\frac{\N(\aaa_\vg+\sum_{d=j}^D\sum_{i=1}^{t_d}\aaa_{d,i})}{\N\aaa_\vg}\right\}^{1/s_j}.
  \end{equation*}
\end{lemma}

\begin{proof}
  This generalizes \cite[Lemma 8.2]{BrowningHeathBrown2014}. The proof
  is essentially the same. We choose $\va = \vg$, $\vt=\mathbf{0}$ in
  Lemma \ref{lem:major_arcs_S}.  Clearly, $\va=\vg$ is in $\M_\vg$ for
  $P = \N\aaa_\vg^A$, with any large fixed value of $A$. Since
  $J(\mathbf{0})\gg 1$, we obtain
  \begin{equation*}
    \Sigma(\vg)\ll \frac{\N\aaa_\vg^s\cdot \abs{S(\vg)}}{P^{ns}}+ \frac{\N\aaa_\vg^{s+1}}{P}.
  \end{equation*}
  We choose $A > s+1$ to obtain
  \begin{equation}
    \label{eq:major_arcs_sigma_estimate}
    \Sigma(\vg) \ll 1 + \N\aaa_\vg^s L,
  \end{equation}
  with $L$ defined via $\abs{S(\va)}=P^{ns}L$ as earlier in the
  paper. Let us apply Lemma \ref{lem:iterativeargument} to estimate
  $L$. If $\vg\in I_d^{(1)}$ for some $d \in \Delta$ then
  \eqref{iterativeestimate} and \eqref{eq:major_arcs_sigma_estimate}
  show that
  \begin{equation*}
    \Sigma(\vg)\ll 1 + \frac{N\aaa_\vg^sP^\epsilon}{P^{n(s-B_d)/(2^{d-1}+(s-B_d)s_{d+1})}}\ll 1
  \end{equation*}
  if $A$ is chosen large enough. Now let us assume that $\va=\vg \in
  I^{(2)}$. In this case, Lemma \ref{lem:iterativeargument} yields
  $q_j \in \n$ and $\vn_j\in\n^{t_j}$ for $j \in \Delta$ satisfying
  \eqref{eq:q_j_divisibility}, \eqref{eq:q_j_bound}, and
  \eqref{eq:q_j_approx}, with $Q_j$ given in
  \eqref{eq:Q_j_def}. Assume first that $q_j\gamma_{j,i}\neq
  \nu_{j,i}$ for some $j,i$. By Minkowski's convex body theorem, there
  is a $q \in \aaa_\vg\smallsetminus\{0\}$ with $\abs{q}\ll\N\aaa_\gamma^{1/n}$. Then $q(q_j\gamma_{j,i}-\nu_{j,i})\in\n$, and so
  \begin{equation*}
    1\ll \abs{q(q_j\gamma_{j,i}-\nu_{j,i})}\ll \N\aaa_{\gamma}^{1/n}Q_jP^{-j}\ll \N\aaa_\vg^{1/n}L^{-s_j/n}P^{-j+\epsilon}.
  \end{equation*}
  This gives an upper bound for $L$, and substituting this bound in
  \eqref{eq:major_arcs_sigma_estimate} shows that $\Sigma(\vg)\ll 1$
  as long as we have chosen $A$ big enough. Hence, we are left with the case where
  \begin{equation}\label{eq:major_arcs_equality}
    q_j\gamma_{j,i} = \nu_{j,i} \quad\text{ for all }\quad j\in \Delta, 1\leq i \leq t_j.
  \end{equation}
  Since $\nu_{j,i}\in\n$, we find integral ideals $\bbb_{j,i}$ such
  that $\nu_{j,i}\Ok = \n\bbb_{j,i}$. After cancellation,
  \eqref{eq:major_arcs_equality} gives
  \begin{equation*}
    q_j \aaa_{j,i} = \aaa_\vg \bbb_{j,i} \quad\text{ for all }\quad j\in \Delta, 1\leq i \leq t_j.
  \end{equation*}
  In the following arguments, we write $\aaa^{(j)} := \aaa_{j,1} + \cdots + \aaa_{j,t_j}$ and $\bbb^{(j)} := \bbb_{j,1} + \cdots + \bbb_{j,t_j}$. Then $q_j\aaa^{(j)} = \aaa_\gamma \bbb^{(j)}$ holds for all $j \in \Delta$. With
  \begin{equation*}
    \ddd_j := \aaa_\vg+\aaa^{(j)} \text{ and }\fff_j := q_j\Ok + \bbb^{(j)},
  \end{equation*}
  we have thus
\begin{equation*}
  \frac{q_j\Ok}{\fff_j}\cdot\frac{\aaa^{(j)}}{\ddd_j} = \frac{\aaa_\vg}{\ddd_j}\cdot \frac{\bbb^{(j)}}{\fff_j}.
\end{equation*}
 We claim that
 \begin{equation}\label{eq:major_arcs_claim}
   \frac{\aaa_\vg}{\sum_{d>j}\ddd_d} \text{ divides }q_j\Ok\text{ for all }j \in \Delta.
 \end{equation}
 Indeed, since $\frac{q_j\Ok}{\fff_j}+\frac{\bbb^{(j)}}{\fff_j}=\Ok$,
 we see that $\frac{q_j\Ok}{\fff_j} \mid \frac{\aaa_\vg}{\ddd_j}$. The
 opposite divisibility follows analogously, and hence
 \begin{equation}
 \frac{q_j\Ok}{\fff_j} = \frac{\aaa_\vg}{\ddd_j}.\label{eq:major_arcs_ideal_equality} 
\end{equation}
Let $k \in \Delta$, $k>j$. Since \eqref{eq:major_arcs_ideal_equality} holds for $k$ as well as for $j$, we see that $q_j\ddd_j\fff_k = q_k\ddd_k\fff_j$. By \eqref{eq:q_j_divisibility}, we can write $q_j = q_k \hat{q_j}$ with $\hat{q_j}\in \Ok$. Substituting this in the above equality, cancelling $q_k\Ok$, and dividing both sides by $\ddd_j + \fff_j$ shows that
\begin{equation*}
  \hat{q_j}\fff_k\frac{\ddd_j}{\ddd_j+\fff_j} = \ddd_k\frac{\fff_j}{\ddd_j+\fff_j},
\end{equation*}
and in particular
\begin{equation}\label{eq:major_arcs_divisibility}
  \frac{\ddd_j}{\ddd_j+\fff_j} \text{ divides }\ddd_k \text{ for all }k>j.
\end{equation}
Let $\delta_j := \sum_{d>j}\ddd_d$. By \eqref{eq:major_arcs_ideal_equality},
\begin{equation*}
  q_j\Ok = \frac{\aaa_\vg}{\delta_j}\cdot\frac{\fff_j\delta_j}{\ddd_j} = \frac{\aaa_\vg}{\delta_j}\cdot \frac{\fff_j}{\ddd_j+\fff_j}\cdot \frac{\delta_j}{\ddd_j(\ddd_j+\fff_j)^{-1}}.
\end{equation*}
By \eqref{eq:major_arcs_divisibility}, the second and the third factor on the right-hand side are integral ideals, and hence \eqref{eq:major_arcs_claim} holds as claimed. Hence,
\begin{equation*}
  \frac{\N\aaa_\vg}{\N(\aaa_\vg+\sum_{d=j}^D\sum_{i=1}^{t_d}\aaa_{d,i})} = \frac{\N\aaa_\vg}{\N\delta_j} \leq \N(q_j\Ok) \ll \abs{q_j}^n \ll Q_j^n\ll L^{-s_j}(\log P)^{n e(j)}.
\end{equation*}
This gives an upper bound for $L$ which, once substituted into \eqref{eq:major_arcs_sigma_estimate}, proves the lemma.
\end{proof}

\begin{lemma}\label{lem:major_arcs_counting}
  Let $\aaa$ be a fractional ideal of $K$. Then
  \begin{equation*}
    |\R\cap\aaa|\ll \frac{1}{\N\aaa}+1.
  \end{equation*}
\end{lemma}

\begin{proof}
  There is a constant $c$ depending only on $K$ and our basis
  $\omega_1, \ldots, \omega_n$ such that $\R\cap\aaa \subseteq \{x \in
  \aaa \mid \abs{x^{(j)}}\leq c \text{ for all } v\}$. Here, the
  $x^{(v)}$ are all the (real and complex) conjugates of $x$. The
  result then follows from \cite[Lemma 7.1]{freipieropan2014}. 
\end{proof}

We are now ready to treat our singular series under the hypotheses of
Theorem \ref{thm:main}.

\begin{lemma}\label{lem:singularseries}
  Assume that
  \begin{equation}
    \label{eq:degree_condition}
    s_1 + \sum_{j=1}^Ds_jt_j<1.
  \end{equation}
  Then the series defining $\mathfrak{S}$ converges absolutely and
  there is a positive constant $\delta$ such that
  \begin{equation*}
    \mathfrak{S}-\mathfrak{S}(H) \ll H^{-\delta}    
  \end{equation*}
  holds for large enough $H$.
\end{lemma}

\begin{proof}
  We write
  \begin{equation*}
    A(\aaa) := \sum_{\substack{\vg\in(\R\cap\K)^T\\\aaa_\vg=\aaa}}|\Sigma(\vg)|.
  \end{equation*}
  For each $\vg\in(\R\cap\K)^T$ with $\aaa_\vg=\aaa$, we write again
  $\gamma_{d,j}\Ok = \frac{\n\aaa_{d,j}}{\aaa}$ and define
  \begin{equation}
    \ddd_j =
    \aaa + \sum_{d=j}^D\sum_{i=1}^{t_d}\aaa_{d,i}.\label{eq:major_arcs_dj}
\end{equation}
Then for $j_0:=\min\Delta$ we have $\ddd_{j_0}=\Ok$. By Lemma
\ref{lem:major_arcs_sigma_estimate}, we see that
  \begin{equation*}
    \Sigma(\vg)\ll\N\aaa^{s+\epsilon/2}\min_{j\in\Delta}\left\{\frac{\N\ddd_j}{\N\aaa}\right\}^{1/s_j} \leq \N\aaa^{s+\epsilon/2} \prod_{j\in\Delta}\left(\frac{\N\ddd_j}{\N\aaa}\right)^{\lambda_j/s_j}, 
  \end{equation*}
  for any $\lambda_j\geq 0$ with $\sum_{j\in\Delta}\lambda_j=1$. As in
  \cite[Section 8]{BrowningHeathBrown2014}, we choose
  \begin{equation}\label{dl:eq8.8}
    \lambda_j :=
    \begin{cases}
      \theta + t_{j_0}s_{j_0} &\text{ if }j=j_0\\
      t_js_j &\text{ if }j\in\Delta\smallsetminus\{j_0\},
    \end{cases}
  \end{equation}
 where $\theta = 1-\sum_{j\in\Delta}t_js_j \in (s_1,1)$. Hence,
  \begin{equation*}
    A(\aaa)\ll\N\aaa^{s+\epsilon/2}\sum_{\substack{\ddd_j \mid \aaa\\j\in \Delta}}m((\ddd_j)_{j\in \Delta}, \aaa)\left(\frac{1}{\N\aaa}\right)^{\theta/s_{j_0}}\prod_{j\in\Delta}\left(\frac{\N\ddd_j}{\N\aaa}\right)^{t_j},
  \end{equation*}
  where $m((\ddd_j)_{j\in\Delta},\aaa)$ is the number of all
  $\vg\in(\R\cap\K)^T$ with $\aaa_\vg=\aaa$ and
  \eqref{eq:major_arcs_dj}. Clearly, any $\vg$ with $\aaa_\vg=\aaa$
  and \eqref{eq:major_arcs_dj} satisfies
  \begin{equation*}
    \gamma_{j,i}\in \frac{\n\ddd_j}{\aaa} \text{ for all } j\in\Delta, 1\leq i \leq t_j.
  \end{equation*}
   Using this and Lemma \ref{lem:major_arcs_counting},
  \begin{align*}
    m((\ddd_j)_{j\in\Delta},\aaa)\leq \prod_{j\in\Delta}\abs{R \cap
      \frac{\n\ddd_j}{\aaa}}^{r_j}\ll
    \prod_{j\in\Delta}\left(\frac{\N\aaa}{\N\ddd_j}\right)^{t_j}.
  \end{align*}
  Hence,
  \begin{equation*}
    A(\aaa)\ll \N\aaa^{s+\epsilon/2-\theta/s_{j_0}}\sum_{\substack{\ddd_j\mid\aaa\\j\in\Delta}}1 \ll \N\aaa^{s-\theta/s_{j_0}+\epsilon}.
  \end{equation*}
  Since $\theta > s_1 = s_{j_0}$, this shows that $\mathfrak{S}$
  converges absolutely and that $\mathfrak{S}-\mathfrak{S}(H)\ll
  H^\delta$ for some appropriate delta.
\end{proof}

\section{Major arcs: singular integral}\label{sec:singint}

Throughout this section, we will assume
\eqref{eq:degree_condition}. For $\vg = (\gamma_{d,i})_{d,i} \in V^T$,
we write $\vg_d := (\gamma_{d,1}, \ldots, \gamma_{d,t_d})$.

\begin{lemma}\label{lem:singint_J_estimate}
  Let $a \in [0, n]$ and $\epsilon>0$. For any $\vg \in
  V^T$, we have
  \begin{equation}\label{eq:J_trivial_estimate}
    J(\vg)\ll 1.
  \end{equation}
  Assume that $\abs{\vg} \geq 1$ and let $d \in \Delta$ with $\vg_d \neq
  0$. Then there exists a unit $u_d \in \units$ such that
  \begin{enumerate}
  \item $\abs{u_d} \ll \abs{\vg}^{a/n}$,
  \item $J(\vg) \ll \abs{\vg}^\epsilon \abs{u_d\vg_d}^{-a/s_d}$.
  \end{enumerate}
  Moreover, we have
  \begin{equation}\label{eq:J_estimate_without_u}
    J(\vg) \ll \abs{\vg}^\epsilon\abs{\vg_d}^{-1/s_d}.
  \end{equation}
\end{lemma}

\begin{proof}
  It is clear that $J(\vg)\ll 1$ holds for all $\vg \in V^T$. Let
  $d\in\Delta$ and assume that $\abs{\vg}\geq 1$ and $\vg_d \neq
  0$. We apply Lemma \ref{lem:major_arcs_S} with $\va :=
  (P^{-j}\gamma_{j,i})_{j,i}$ and $P := \abs{\vg}^A$ for fixed large
  $A$. Clearly, $\va \in \M_{\mathbf{0}}$ as soon as $A \geq
  1/\varpi$. Since $\Sigma(\mathbf{0}) = 1$, we
  obtain
  \begin{equation}\label{eq:sing_int_J_estimate_1} J(\vg) \ll L
    + \abs{\vg}P^{-1}\ll L + \abs{\vg}^{-a/s_d},
  \end{equation} if $A$ was chosen big enough.

  If $L \leq \abs{\vg}^{-a/s_d+\epsilon}$ then
  \eqref{eq:sing_int_J_estimate_1} yields
  \begin{equation*}
    J(\vg) \ll \abs{\vg}^\epsilon\abs{\vg}^{-a/s_d} \leq \abs{\vg}^\epsilon\abs{\vg_d}^{-a/s_d},
  \end{equation*}
  and we can choose $u_d=1$. Therefore, we may assume from now on that
  \begin{equation}\label{eq:L_lower_bound_assumption}
    L \geq \abs{\vg}^{-a/s_d+\epsilon},
  \end{equation}
  so that
  \begin{equation}\label{eq:sing_int_J_estimate_2}
    J(\vg)\ll L.
  \end{equation}
  The remainder of this proof is devoted to the deduction of suitable
  upper bounds for $L$. Let us first assume that $\va\in I_j^{(1)}$
  for some $j \in \Delta$. Then the definition of $I_j^{(1)}$, see
  \eqref{eq:L_bound_j}, yields an upper bound
  \begin{equation*} 
     L\ll \abs{\vg}^{-a/s_d+\epsilon}\leq \abs{\vg}^\epsilon\abs{\vg_d}^{-a/s_d},
   \end{equation*} provided that we have chosen $A$ big enough to
   ensure that $A(s-B_j)>a(2^{j-1}+s_{j+1}(s-B_j))/(ns_d)$.

   If $\va\in I^{(2)}$ then Lemma \ref{lem:iterativeargument}
   yields $q_d \in \n$, and $\vn_d\in\n^{t_d}$ satisfying
   \begin{align}
       \abs{q_d}&\leq Q_d&& &\label{eq:singint_q_d_bound}\\
  \abs{q_d\alpha_{d,i} - \nu_{d,i}}&\leq Q_dP^{-d}&& &\text{ for all
  } 1\leq i\leq t_d,\label{eq:singint_q_d_approx} 
   \end{align}
   with $Q_d$ defined by \eqref{eq:Q_j_def}.

   Suppose that $\nu_{d,i}\neq 0$ for some $1\leq i
   \leq t_d$. Then
   \begin{equation*} 1\leq \abs{\nu_{d,i}} \ll Q_d
     P^{-d}\abs{\gamma_{d,i}} + Q_dP^{-d}\ll Q_dP^{-d}|\vg|=
     L^{-s_d/n}(\log P)^{e(d)}P^{-d}|\vg|.
   \end{equation*}
   This yields an upper bound
   \begin{equation*}
     L \ll (\log P)^{n e(d)/s_d}
     P^{-nd/s_d}\abs{\vg}^{n/s_d}\ll \abs{\vg}^\epsilon
     \abs{\vg}^{\frac{n}{s_d}(1-dA)} \ll \abs{\vg}^\epsilon
     \abs{\vg}^{-a/s_d}\ll \abs{\vg}^\epsilon
     \abs{\vg_d}^{-a/s_d}
   \end{equation*}
   if $A$ is chosen big enough. We are left with the case where $\vn_d
   = \mathbf{0}$. Let $t$ be a generator of the principal ideal $q_d\Ok$
   with the property that
  \begin{equation}\label{eq:t_conjugates}
     N(q_d)^{1/n} \ll \absv{t} \ll N(q_d)^{1/n}\text{ for all }v\in\archplaces,
   \end{equation}
   and let $u_d := q_d/t \in \units$. Thanks to
   \eqref{eq:t_conjugates}, \eqref{eq:singint_q_d_bound} and
   \eqref{eq:L_lower_bound_assumption}, we obtain
   \begin{equation*}
     \abs{u_d} \asymp N(q_d)^{-1/n}\abs{q_d}\leq Q_d = L^{-s_d/n}(\log P)^{e(d)} \ll \abs{\vg}^{a/n}.
   \end{equation*}
   Moreover, due to \eqref{eq:singint_q_d_approx}, for all $1\leq i
   \leq t_d$ we have
   \begin{equation}\label{eq:critical_L_estimate}
     P^{-d}\abs{u_d\vg_{d}} =
     \abs{u_d\va_{d}}\asymp N(q_d)^{-1}\abs{q_d\va_{d}}\leq Q_d P^{-d} = L^{-s_d/n}(\log P)^{e(d)}P^{-d},
  \end{equation}
  so
  \begin{equation*}
    L \ll \abs{\vg}^{\epsilon}\abs{u_d\vg_d}^{-n/s_d}.
  \end{equation*}
  The estimate $J(\vg)\ll\abs{\vg}^\epsilon\abs{u_d\vg_d}^{-a/s_d}$
  follows immediately from this and \eqref{eq:sing_int_J_estimate_2}
  if $\abs{u_d\vg_d}\geq 1$, and from the trivial estimate $J(\vg)\ll
  1$ if $\abs{u_d\vg_d}\leq 1$.

  For the proof of \eqref{eq:J_estimate_without_u}, we proceed as
  above with $a=1$, until \eqref{eq:critical_L_estimate}. Here, we conclude that
  \begin{equation*}
    P^{-d}\abs{\vg_{d}} = \abs{\va_{d}}\ll \abs{q_d^{-1}}\abs{q_d\va_{d}} \ll Q_d^{n-1}\cdot Q_dP^{-d} = L^{-s_d}(\log P)^{n e(d)}P^{-d},
  \end{equation*}
  and thus 
  \begin{equation*}
    L \ll \abs{\vg}^\epsilon\abs{\vg_{d}}^{-1/s_d}.
  \end{equation*}
\end{proof}
  We write
  \begin{equation*}
    N(\vg_d) := \prod_{v\in\archplaces}\absv{\vg_d}^{n_v},
  \end{equation*}
  where
  $\absv{\vg_d}:=\max\{\absv{\gamma_{d,1}},\ldots,\absv{\gamma_{d,t_d}}\}$
  and $n_v := [K_v : \RR]$ is the local degree. Let $b \in (0,1)$ be a
  constant to be specified later, and for $H \geq 1$ let
  \begin{align*}
   M(H) &:= \{\vg \in V^T \where \abs{\vg} \gg H\},\\
    M_{>}(H) &:= \{\vg \in V^T \where \abs{\vg}\in (H, 2H] \text{ and there exists $d$ with }N(\vg_{d}) \geq H^b\},\\
   M_{<}(H) &:= \{\vg \in V^T \where \abs{\vg}\in (H, 2H] \text{ and for all $d$, we have }N(\vg_{d}) \leq H^b\}.
 \end{align*}
  Define the integral
  \begin{equation*}
    I(H) := \int_{M(H)}\max_{d\in\Delta}\{\abs{\vg_d}^{n/s_d}\}^{-1}\dd\vg.
  \end{equation*}
 
 \begin{lemma}\label{lem:singint_without_u}
   There is $\delta > 0$ such that, for $H \geq 1$,
   \begin{equation*}
     I(H) \ll H^{-\delta}.
   \end{equation*}
 \end{lemma}

 \begin{proof}
   We identify $V^T$ with $\RR^{nT}$ using the basis $\omega_1,
   \ldots, \omega_n$ of $V$. The exponent $n/s_d$ in the definition
   $I(H)$ is good enough for the arguments given after \cite[Lemma
   8.3]{BrowningHeathBrown2014} to apply.
 \end{proof}

  For $\vu=(u_d)_{d\in\Delta}$ with $u_d \in \units$ for all $d\in\Delta$, let
  \begin{equation*}
    I_{>}(\vu, H) := \int_{M_{>}(H)}\max_{\substack{d\in\Delta}}\{\abs{u_d\vg_d}^{n/s_d}\}^{-1}\dd\vg.
  \end{equation*}

  \begin{lemma}\label{lem:singint_normbig}
    Let $\delta$ be as in Lemma \ref{lem:singint_without_u}. Then, for $H \geq 1$,
    \begin{equation*}
      I_{>}(\vu,H)\ll H^{-\delta b/n}.
    \end{equation*}
  \end{lemma}

  \begin{proof}
    Let $\phi : V^T \to V^T$ be the $\RR$-linear transformation
    $(\gamma_{d,i})_{d,i}\mapsto (u_d\gamma_{d,i})_{d,i}$. Since the
    $u_d$ are all units, we have $\det \phi \asymp 1$. Moreover, let
    $\vg \in M_{>}(H)$ and $d$ such that $N(u_d\vg_{d})=N(\vg_{d})\geq
    H^b$. Then in particular
    \begin{equation*}
      \max_{v\in\archplaces}\{\lvert u_d\vg_{d}\rvert_v\} \geq H^{b/n},
    \end{equation*}
    and thus $\abs{\phi(\vg)}\gg H^{b/n}$. We have shown that $\phi(M_{>}(H)) \subseteq M(H^{b/n})$. By Lemma \ref{lem:singint_without_u}, we obtain
    \begin{equation*}
      I_{>}(\vu,H) \asymp \int_{\phi(M_{>}(H))}\max_{d\in\Delta}\{\abs{\vn_d}^{n/s_d}\}^{-1}\dd\vn \ll \int_{M(H^{b/n})}\max_{d\in\Delta}\{\abs{\vn_d}^{n/s_d}\}^{-1}\dd\vn = I(H^{b/n})\ll H^{-b\delta/n}.
    \end{equation*}
  \end{proof}

Let
\begin{equation*}
  I_{<}(H) := \int_{M_{<}(H)}\max_{d\in\Delta}\{\abs{\vg_d}^{1/s_d}\}^{-1}\dd\vg.
\end{equation*}
We will prove that $I_{<}(H)\ll H^{-\delta}$ for some $\delta>0$.

\begin{lemma}\label{lem:elementary_integral}
  For $A_1, \ldots, A_n, B \in (0,\infty)$, let
  \begin{equation*}
    I(A_1, \ldots, A_n, B) := \int_{\substack{0\leq x_i \leq A_i\\x_1\cdots x_n \leq B}}\dd x_1 \cdots \dd x_n.
  \end{equation*}
  Then
  \begin{equation*}
    I(A_1, \ldots, A_n, B) \ll B\log\left(\frac{A_1\cdots A_n}{B} + 2\right)^{n-1}.
  \end{equation*}
\end{lemma}

\begin{proof}
 Elementary computations using induction.
\end{proof}

\begin{lemma}\label{lem:singint_integral_large}
  Let $\theta > 0$. For any small $\epsilon> 0$, and $H \geq 1$, we
  have
  \begin{equation*}
    \int_{\substack{\gamma\in V\\N(\gamma)\leq H^b\\\abs{\gamma}\geq H}}\frac{1}{(1+\abs{\gamma})^{1+\theta}}\dd \gamma \ll H^{-1-\theta + \epsilon + b}.
  \end{equation*}
\end{lemma}

\begin{proof}
  For $v\in\archplaces$, let $t_v := \absv{\gamma}^{n_v}$, and assume
  that $\max_{v}\absv{\gamma} = \absv{\gamma}$. Passing to
  polar coordinates at the complex places, we see that the integral in the lemma is
  \begin{equation*}
    \ll \int_{t_{w}\gg H^{n_{w}}}\frac{1}{t_{w}^{(1+\theta)/n_w}}\left(\int_{\substack{t_v, v\neq w\\t_v\leq t_w^{n_v/n_w}\\\prod_{v\neq w}t_v\leq H^b/t_w}}\prod_{v\neq w}\dd t_v \right)\dd t_w.
  \end{equation*}
  Using the notation of Lemma \ref{lem:elementary_integral}, the inner
  integral is just
  \begin{equation*}
    I((t_w^{n_v/n_w})_{v\neq w}, H^b/t_w) \ll \frac{H^b}{t_w^{1-\epsilon}},
  \end{equation*}
  and thus the integral in the lemma is
  \begin{equation*}
    \ll H^b\int_{t_w\gg H^{n_w}}\frac{1}{t_w^{(1+\theta)/n_w+1-\epsilon}}\dd t_w \ll H^{b-1-\theta+n_w\epsilon}.
  \end{equation*}
\end{proof}

\begin{lemma}\label{lem:singint_integral_small}
  Let $\theta > 0$. For any small $\epsilon> 0$, we have
  \begin{equation*}
    \int_{\substack{\gamma\in V\\N(\gamma)\leq H^b}}\frac{1}{(1+\abs{\gamma})^{1+\theta}}\dd \gamma \ll H^{b}.
  \end{equation*}
\end{lemma}

\begin{proof}
  We start as in the proof of Lemma \ref{lem:singint_integral_large} and see that the integral is
  \begin{align*}
    &\ll \int_{t_w=0}^{\infty}\frac{1}{(1+t_w^{1/n_w})^{(1+\theta)}}\cdot  I((t_w^{n_v/n_w})_{v\neq w}, H^b/t_w)\dd t_w.\\
&\ll \int_{t_w=0}^{H^{bn_w/n}}t_w^{\sum_{v\neq w}n_v/n_w}\dd t_w + H^b\int_{t_w = H^{bn_w/n}}^\infty\frac{1}{t_w^{(1+\theta)/n_w+1-\epsilon}}\dd t_w\\
&\ll H^b + H^{b(1-1/n -\theta/n +\epsilon n_w/n)} \ll H^b.
  \end{align*}
\end{proof}

\begin{lemma}\label{lem:singint_norm_small_estimate}
  Assume that $b \leq 1/T$. Then there is
  $\delta>0$ such that, for $H \geq 1$,
  \begin{equation*}
    I_{<}(H) \ll H^{-\delta}.
  \end{equation*}
\end{lemma}

\begin{proof}
  Let $d_0\in\Delta$ and assume that $\abs{\vg} =
  \abs{\vg_{d_0}}\in(H, 2H]$. Since $\abs{\vg} > H \geq 1$, we have
  \begin{equation*}
    \max_{d\in\Delta}\{\abs{\vg_d}^{1/s_d}\} \gg \max_{d\in\Delta}\{(1+\abs{\vg_d})^{1/s_d}\} \geq \prod_{d\in\Delta}(1+\abs{\vg_d})^{\lambda_d/s_d}
  \end{equation*}
  for any choice of $0 \leq \lambda_d \leq 1$ with
  $\sum_{d\in\Delta}\lambda_d = 1$. We choose
  \begin{equation*}
    \lambda_d := s_d t_d + \theta/\abs{\Delta},
  \end{equation*}
  where $\theta := 1 - \sum_{d\in\Delta}s_dt_d \in (0,1)$. With $\theta_d := \theta/(\abs{\Delta}s_dt_d)$, this gives
  \begin{equation*}
    I_{<}(H)\ll \int_{\substack{\vg_{d_0}\in V^{t_{d_0}}\\N(\vg_{d_0})\leq H^b\\\abs{\vg_{d_0}}\geq H}}\frac{1}{(1+\abs{\vg_{d_0}})^{t_{d_0}(1 + \theta_{d_0})}}\dd\vg_{d_0}\prod_{\substack{d\in\Delta\\d\neq d_0}}\int_{\substack{\vg_d\in V^{t_d}\\N(\vg_d)\leq H^b}}\frac{1}{(1+\abs{\vg_d})^{t_d(1+\theta_d)}}\dd\vg_d.
  \end{equation*}
  Further assuming that $\abs{\vg_{d_0}} = \abs{\vg_{d_0,i_0}}$, we get
  \begin{equation*}
    I_{<}(H)\ll \int_{\substack{\gamma_{d_0,i_0}\in V\\N(\gamma_{d_0,i_0})\leq H^b\\\abs{\gamma_{d_0,i_0}}\geq H}}\frac{1}{(1+\abs{\gamma_{d_0,i_0}})^{1 + \theta_{d_0}}}\dd\gamma_{d_0,i_0}\prod_{\substack{d\in\Delta\\1\leq i \leq t_d\\(d,i)\neq (d_0,i_0)}}\int_{\substack{\gamma_{d,i}\in V\\N(\gamma_{d,i})\leq H^b}}\frac{1}{(1+\abs{\gamma_{d,i}})^{1+\theta_d}}\dd\gamma_{d,i}.
  \end{equation*}
  By Lemma \ref{lem:singint_integral_large} and Lemma
  \ref{lem:singint_integral_small}, this product is $\ll
  H^{-1-\theta_{d_0}+\epsilon+Tb} \ll
  H^{-\theta_{d_0}/2}$ if we choose $\epsilon$ small enough.
\end{proof}

With all our auxiliary results in place, we can now proceed to our
main task, the estimation of the singular integral $\mathfrak{J}$.

\begin{lemma}\label{lem:singularintegral}
  Assume \eqref{eq:degree_condition}. Then the integral defining
  $\mathfrak{J}$ converges absolutely and there is a positive constant
  $\delta$ such that
  \begin{equation*}
    \mathfrak{J} - \mathfrak{J}(H) \ll H^{-\delta}   
  \end{equation*}
  holds for all large enough $H$.
\end{lemma}

\begin{proof}
  Let us fix $b := 1/T$. We have
  \begin{align*}
    \mathfrak{J}-\mathfrak{J}(H) &\ll \int_{\abs{\vg}> H}\abs{J(\vg)}\dd\vg =\sum_{j=0}^\infty\int_{2^jH< \abs{\vg}\leq 2^{j+1}H}\abs{J(\vg)}\dd\vg\\
    &\ll \sum_{j=0}^{\infty}\left(\int_{M_{<}(2^jH)}\abs{J(\vg)}\dd\vg
      + \int_{M_{>}(2^j H)}\abs{J(\vg)}\dd\vg\right).
  \end{align*}
  We first consider the integrals over $M_{<}(2^jH)$. Here, we
  estimate $\abs{J(\vg)}$ by \eqref{eq:J_estimate_without_u} and
  obtain
  \begin{equation*}
    \sum_{j=0}^\infty\int_{M_{<}(2^jH)}\abs{J(\vg)}\dd\vg \ll\sum_{j=0}^{\infty}(2^{j}H)^\epsilon I_{<}(2^jH)\ll H^{\epsilon-\delta}\sum_{j=0}^\infty 2^{j(\epsilon-\delta)} \ll H^{-\delta/2},
  \end{equation*}
  by Lemma \ref{lem:singint_norm_small_estimate}, if $\epsilon$ was
  chosen small enough.

  For the integrals over $M_{>}(2^jH)$, we use the estimates from
  \emph{(1)} and \emph{(2)} in Lemma \ref{lem:singint_J_estimate} with
  $a = n$. We obtain
  \begin{align*}
    \sum_{j=0}^\infty\int_{M_{>}(2^jH)}\abs{J(\vg)}\dd\vg \ll
    \sum_{j=0}^\infty(2^jH)^\epsilon\sum_{\substack{\vu =
        (u_d)_{d\in\Delta}\\u_d\in\units\\\abs{u_d}\ll
        2^jH}}I_{>}(\vu, 2^jH).
  \end{align*}
  By Lemma \ref{lem:singint_normbig}, we have $I_{>}(\vu, 2^j) \ll
  (2^{j}H)^{-\delta}$. Moreover, it is well known that the number of
  units $u \in \units$ with $\abs{u}\ll 2^jH$ is is $\ll \log(2^j
  H)^{\abs{\archplaces}-1}$. Hence, the inner sum in the above
  expression has $\ll (2^jH)^\epsilon$ summands. Altogether, we see
  that
  \begin{equation*}
    \sum_{j=0}^\infty\int_{M_{>}(2^jH)}\abs{J(\vg)}\dd\vg \ll H^{2\epsilon-\delta}\sum_{j=0}^\infty2^{j(2\epsilon-\delta')} \ll H^{-\delta/2},
  \end{equation*}
  if $\epsilon$ was chosen small enough.
\end{proof}

Our Theorem \ref{thm:main} is now an immediate consequence of the estimation of the minor arcs in Section \ref{sec:minorarcs} and the treatment of the major arcs in Lemma \ref{lem:major_arcs_asympt}, Lemma \ref{lem:singularseries}, and Lemma \ref{lem:singularintegral}.

\begin{acknowledgements}
  We would like to thank Prof.~Tim Browning, Prof.~Jörg Brüdern and
  Dr.~Damaris Schindler for helpful discussions, and Prof.~Christopher
  Skinner for useful and encouraging remarks. The first-named author
  was supported by a Humboldt Research Fellowship for Postdoctoral
  Researchers of the Alexander von Humboldt Foundation. Major parts of
  the present work were established when the second-named author was a
  visitor of the Institut f\"ur Algebra, Zahlentheorie und Diskrete
  Mathematik at Leibniz Universit\"at Hannover. He thanks the
  institute for its hospitality.
\end{acknowledgements}

\bibliographystyle{alpha}
\bibliography{bibliography}
\end{document}